\documentclass[11pt]{amsart}
\usepackage{amsfonts}
\usepackage{bbm}
\usepackage{amsfonts,amssymb,amsmath,amsthm}
\usepackage{url}
\usepackage{enumerate}
\usepackage{bbm}
\usepackage[all]{xy}
\usepackage[pdftex, colorlinks, citecolor=red, backref=page]{hyperref}

\newtheorem{thrm}{Theorem}[section]
\newtheorem{lem}[thrm]{Lemma}
\newtheorem{prop}[thrm]{Proposition}
\newtheorem{cor}[thrm]{Corollary}
\theoremstyle{definition}
\newtheorem{definition}[thrm]{Definition}
\newtheorem{remark}[thrm]{Remark}
\numberwithin{equation}{section}

\author{George A. Elliott}
\address{Department of Mathematics, University of Toronto, Toronto, Ontario, Canada \ M5S 2E4}
\email{elliott@math.toronto.edu}
\author{Zhiqiang Li}
\address{College of Mathematics and Statistics, Chongqing University, Chongqing, China \ 401331}
\email{zqli@cqu.edu.cn}
\author{Xia Zhao}
\address{College of Mathematics, Tongji University, Shanghai, China \ 200092}
\email{1710383@tongji.edu.cn}



\begin{document}

\title[Thomsen-Li's theorem revisited ]{Thomsen-Li's theorem revisited}

\begin{abstract}
In this paper, we prove a general version of Thomsen-Li's theorem--a Krein-Milman type theorem for C*-algebras. Given a Markov operator on $\mathrm{C}[0,1]$ preserving a certain type of subspaces, which correspond to certain subhomogeneous C*-algebras on $[0,1]$, we approximate it by an average of $*$-homomorphisms on $\mathrm{C}[0,1]$ in the strong operator topology; moreover, we require the average preserves the same subspace. We also generalize this Krein-Milman type results to the case of path connected compact metrizable spaces. Such results could be useful for constructing $*$-homomorphisms of subhomogeneous C*-algebras.
\end{abstract}

\maketitle
\section{introduction}
There have been already great amounts of results in the classification program of C*-algebras, which is proposed by George A. Elliott, see for example \cite{G}. As it is well known, the main ingredients for the classification program involve two parts, i.e., existence theorems and uniqueness theorems. Among the proof of existence theorems, K. Thomsen and L. Li established some approximation results of a Markov operator by an average of homomorphisms, see \cite{Th} and \cite{Li}; especially,  Li's theorem achieves the improvement that the number of homomorphisms used for the approximation only depends on the given data of a finite subset and a tolerance.

However, all such  results are of concern with homogeneous C*-algebras $\mathrm{C}(X)$, the algebra of continuous functions on some compact metrizable space $X$. There are also many subhomogeneous C*-algebras, i.e., subalgebras of $\mathrm{C}(X)$. In the case of the unit interval, X. Jiang and H. Su studied the typical subhomogeneous algebras: splitting interval algebras and dimension drop interval algebras, see \cite{JS1} and \cite{JS2}. Moreover, there are even natural non-unital subhomogeneous algebras, for example, the building blocks considered in \cite{Rak}, they are stably projectionless. For such algebras, one cannot apply Thomsen and Li' results directly, it is desirable to establish some non-unital version of their theorems. In this paper, we will try to do so.

To be precise, we investigate the approximation problem for a Markov operator on homogeneous algebras preserving certain subspace, which indicates certain subhomogeneity. We want to use an average of homomorphisms on homogeneous algebras to do an approximation, additionally we require that the average also preserves the same subspace. In this paper, first of all we present such an approximation for $\mathrm{C}[0,1]$ with a subspace to be specified from certain C*-algebras. Consider the following C*-algebra (see e.g. \cite{Rak}):
$$\begin{aligned}
B(a,k)=\Big\{&f\in \mathrm{C}[0,1]\otimes \textrm{M}_m\ |\ f(0)=\left(
\begin{array}{cc}
                                             W\otimes id_a &    \\
                                              
                                                & 0\otimes id_{k}   \\
                                           \end{array}
                                         \right),\\
&f(1)=W\otimes id_{a+k}, W\in \textrm{M}_n\Big\},
\end{aligned} $$
where $a$ and $k$ are natural numbers, and $m$ and $n$ are proper integers. Such a C*-algebra is non-unital and stably projectionless. The space of continuous affine functions on tracial cone of $B(a,k)$ is isomorphic to the subspace $\mathrm{C}[0,1]_{(a,k)}$ of $\mathrm{C}[0,1]$ given by: $$\mathrm{C}[0,1]_{(a,k)}=\{f\in \mathrm{C}[0,1]\mid f(0)=\frac{a}{a+k}f(1)\},$$ (see Proposition 2.1 of \cite{Rak}). 

The first main result of this paper is the following theorem:

\begin{thrm}\label{thm1}
Given any finite subset $ F \subset C[0,1]_{(a,k)}$ and $\epsilon>0$, there is an integer $ N>0$ with the following property:
for any unital positive linear map $\phi$ on $C[0,1]$ which preserves $C[0,1]_{(a,k)}$, there are N homomorphisms $\phi_{1},\phi_{2},...,\phi_{N}$ from $C[0,1]$ to $C[0,1]$ such that $\dfrac{1}{N}\sum\limits_{i=1}^{N}\phi_{i}(f)$ belongs to $C[0,1]_{(a,k)}$ for all $f\in C[0,1]_{(a,k)}$ and
$$\|\phi(f)-\frac{1}{N}\sum\limits_{i=1}^N\phi_{i}(f)\| < \epsilon$$ for all $f\in F$.
\end{thrm}

 For the case of different subspaces $\mathrm{C}[0,1]_{(a,k)}$ and $\mathrm{C}[0,1]_{(b,k)}$, there is no unital positive linear map on $\mathrm{C}[0,1]$ which sends $\mathrm{C}[0,1]_{(a,k)}$ to $\mathrm{C}[0,1]_{(b,k)}$ if $b<a$, see Remark \ref{b<a}. So we only need to deal with the case $b>a$.

\begin{thrm}\label{thm2}
Given any finite subset $ F \subset C[0,1]_{(a,k)}$ and $\epsilon>0$, any unital positive linear map $\phi$ on $C[0,1]$ which sends $C[0,1]_{(a,k)}$ to $C[0,1]_{(b,k)}$ with $b>a$, then there are N homomorphisms $\phi_{1},\phi_{2},...,\phi_{N}: C[0,1]\rightarrow C[0,1]$ such that $\dfrac{1}{N}\sum\limits_{i=1}^N\phi_{i}(f)$ belongs to $C[0,1]_{(b,k)}$ for all $f\in C[0,1]_{(a,k)}$ and
\[ \|\phi(f)-\frac{1}{N}\sum_{i=1}^N\phi_{i}(f)\| < \epsilon\] for all $f\in F$.
\end{thrm}

In general, we replace $\frac{a}{a+k}$ by $\alpha\in(0,1)$ and $[0,1]$ by path connected compact metrizable spaces $X$ and $Y$. Let $p, q$ and $z, w$ be fixed points in $X$ and $Y$ respectively, define $$\mathrm{C}(X)_{(p,q,\alpha)}=\{f\in \mathrm{C}(X) \mid f(p)=\alpha f(q)\},$$ and $$\mathrm{C}(Y)_{(z,w,\beta)}=\{f\in \mathrm{C}(X) \mid f(z)=\beta f(w)\}.$$

Then we have the following generalizations:

\begin{thrm}\label{thm3}
Suppose that $X$ is a path connected compact metrizable space, $Y$ is a compact metrizable space. For any finite subset $ F \subset C(X)_{(p,q,\alpha)}$ and $\varepsilon>0$. There is an integer $ N>0$ with the following property:
for any  unital positive linear map $\phi:C(X)\rightarrow C(Y)$ which sends $C(X)_{(p,q,\alpha)}$ to $ C(Y)_{(z,w,\alpha)}$, then there are N homomorphisms $\phi_{1},\phi_{2},...,\phi_{N}:C(X)\rightarrow C(Y)$ such that 
$\dfrac{1}{N}\sum\limits_{i=1}^{N}\phi_{i}(f)$ belongs to $C(Y)_{(z,w,\alpha)}$ for all $f\in C(X)_{(p,q,\alpha)}$ and
$$\|\phi(f)-\frac{1}{N}\sum\limits_{i=1}^N\phi_{i}(f)\| < \epsilon$$ for all $f\in F$.
\end{thrm}

\begin{thrm}\label{thm4}
Suppose that $X$ is a path connected compact metrizable space, $Y$ is a compact metrizable space, and $\alpha$, $\beta$ are rational numbers with $\beta>\alpha$. For any finite subset $ F \subset C(X)_{(p,q,\alpha)}$ and $\varepsilon>0$, any  unital positive linear map $\phi:C(X)\rightarrow C(Y)$ which sends $C(X)_{(p,q,\alpha)}$ to $ C(Y)_{(z,w,\beta)}$, there are N homomorphisms $\phi_{1},\phi_{2},...,\phi_{N}:C(X)\rightarrow C(Y)$ such that $\dfrac{1}{N}\sum\limits_{i=1}^{N}\phi_{i}(f)$ belongs to $C(Y)_{(z,w,\beta)}$ for all $f\in C(X)_{(p,q,\alpha)}$ and
$$\|\phi(f)-\frac{1}{N}\sum\limits_{i=1}^N\phi_{i}(f)\| < \epsilon$$ for all $f\in F$.
\end{thrm}

To achieve these results, we keep tracking the approximation process of Li's theorem in \cite{Li}, and the crucial point is that we must argue if we are able to choose proper eigenvalue maps to define homomorphisms such that their average preserves the subspace. The existence of such a choice relies on an analysis of the measures induced by evaluations of a given Markov operator at 0 and 1.

The paper is organized as follows. Section 2 contains some preliminaries about Markov operators, and basic properties of the subspace $\mathrm{C}[0,1]_{(a,k)}$ of $\mathrm{C}[0,1]$. In Section 3, concrete analysis of the measures induced by evaluations of a given Markov operator at 0 and 1 are given, and based on this the proof of Theorem \ref{thm1} and \ref{thm2} are presented.
Section 4 gives the general case of the approximation results of Theorem \ref{thm3}, \ref{thm4}.

\section{preliminaries}
\begin{definition} Given compact Hausdorff spaces $X$ and $Y$, a Markov operator $T$ from $C(X)$ to $C(Y)$ is a unital positive linear map.
\end{definition}
The set of Markov operators forms a convex set of bounded operators, and it is well known that the extreme points are the unital homomorphisms, see \cite{Est}.

K. Thomsen and L. Li proved a Krein-Milman type theorem for Markov operators, which is frequently used in the classification program for establishing existence theorems. What we are concerned with is for the case of subhomogeneous algebras, see for example, \cite{JS1}, \cite{JS2}, etc.. In S. Razak's paper \cite{Rak}, he considered
certain stably projectionless building blocks--necessarily non-unital. The space of continuous affine functions on this building block's tracial cone is a non-unital subspace of $C[0,1]$, see Proposition 2.1 in \cite{Rak}. Therefore, we consider Markov operators on homogeneous algebras which preserve this subspace.
Fix a positive integer $a$ and a positive integer $k$, denote by $C[0,1]_{(a,k)}$ the subspace of $C[0,1]$: $$C[0,1]_{(a,k)}=\{f\in C[0,1]\mid f(0)=\frac{a}{a+k}f(1)\}.$$

Next we shall see some examples of Markov operators on $C[0,1]$ which preserve this subspace.
\begin{itemize}
\item Example 1: Choose a continuous function $\lambda: [0,1]\rightarrow [0.1]$ with $\lambda(0)=0$  and $\lambda(1)=1$, define a Markov operator $T$ from $C[0,1]$ to $C[0,1]$ by $T(f)=f\circ \lambda$, then $T$ sends functions in $C[0,1]_{(a,k)}$ to $C[0,1]_{(a,k)} $.
\item Example 2: Choose two continuous functions $\lambda_1$  and $\lambda_2$ on the unit interval with $\lambda_1(0)=0, \lambda_1(1)=1$ and $\lambda_2(0)=1, \lambda_2(1)=0$, and choose proper $k_1>k_2\geq0$, then one can define a Markov operator $T$ on $C[0,1]$ as follows: $$T(f)=\frac{k_1f\circ\lambda_1+k_2f\circ\lambda_2}{k_1+k_2}.$$ Then $T$ sends functions in $C[0,1]_{(a,k)}$ to $C[0,1]_{(b,k)}$, where $b=\dfrac{k_1a+k_2a+k_2k}{k_1-k_2}\geq0$.
\item Example 3: In general, one has similar examples which involve more points. Choose the $\lambda_1, \lambda_2$ above, and $\lambda_3(t)=1/2$ (one can pick up any point in the middle), then one can choose proper $k_1, k_2$ to define a Markov operator $T$ as follows: $$T(f)(t)=\frac{k_1f\circ\lambda_1+k_2f\circ\lambda_2+s(t)f\circ\lambda_3}{k_1+k_2+s(t)},$$ where $s(t)=(\dfrac{k_1a}{k+a}+k_2)(1-t)+(\dfrac{k_2a}{k+a}+k_1)t$. One can verify that $T$ sends functions in $C[0,1]_{(a,k)}$ to $C[0,1]_{(b,k)}$
 for some $b>0$. This process can continue to involve more points in $[0,1]$.
\end{itemize}

The following direct sum decomposition holds.

\begin{lem}\label{lem:direct}
 $
  C[0,1]=C[0,1]_{(a,k)}\oplus  \mathbb{R}.$
\end{lem}
\begin{proof} Suppose $f\in C[0,1]$, let $\lambda=\frac{(a+k)f(0)-af(1)}{k},\;g(x)=f(x)-\lambda$
Then $f(x)=\lambda+g(x)$ and $g(x)\in C[0,1]_{(a,k)}$.
Next we show the decomposition is unique. Assume $f(x)=g_{1}(x)+\lambda_{1}=g_{2}(x)+\lambda_{2}$. Then $ \lambda_{1}-\lambda_{2}=g_{1}(x)-g_{2}(x)$ for any $x\in[0,1]$. It follows that $ \lambda_{1}-\lambda_{2}=g_{1}(0)-g_{2}(0)=g_{1}(1)-g_{2}(1)$ for any $x\in[0,1]$. But $g_{1}(0)-g_{2}(0)=\frac{a}{a+k}(g_{1}(1)-g_{2}(1))$. Since $a\neq0$, we have $\lambda_{1}=\lambda_{2}$ and $g_{1}(x)=g_{2}(x)$.
\end{proof}

In the dual viewpoint, one might consider positive linear maps on these subspaces which extend to Markov operators on $C[0,1]$. Based on the direct sum decomposition, any positive linear map $\phi$ from $C[0,1]_{(a,k)}$ to $C[0,1]$ can be extended naturally to a unital linear map $\widetilde{\phi}$ from $C[0,1]$ to $C[0,1]$, given as $\widetilde{\phi}(f)=\lambda+\phi(g)$. Moreover, this algebraic extension needs to be positive.

\begin{definition}\label{th:lsc}
A positive linear map $\phi$ from $C[0,1]_{(a,k)}$ to $C[0,1]$ is called positively extendible if $\widetilde{\phi}$ is still positive.
\end{definition}

First, it is not hard to see that if $\phi$ is positively extendible, then it must be a contraction. However, we point out that
the converse is not true, even in the case $\phi$ is of norm one.
\begin{remark}
 Fix a $x_{0}\in (0,1)$, define a map $\phi$ from $C[0,1]_{(a,k)}$ to $C[0,1]$ as $\phi(g)(x)=g(x_{0})\frac{a+kx}{a+k}$ for any $g\in C[0,1]_{(a,k)}$. It is obvious that $\phi$ is positive and linear. Moreover, $\phi$ has norm one. First, for any $g\in C[0,1]_{(a,k)}$, $\|\phi(g)\|=\sup_{x\in[0,1]}|\phi(g)(x)|=\sup_{x\in[0,1]}|g(x_{0})\frac{a+kx}{a+k}|\leq\|g\|$. Next we choose a function for which the supreme reaches one. Define a function $g_1(x)$ as follows
 $$g_{1}(x)=\begin{cases}
                 \frac{a}{a+k}+\frac{(1-\frac{a}{a+k})x}{x_{0}} & x\in[0,x_{0}] \\
                  1 & x\in(x_{0},1], \\
               \end{cases}$$
    then $\|\phi(g_{1})\|=sup_{x\in[0,1]}|\phi(g_{1})(x)|=sup_{x\in[0,1]}|g_{1}(x_{0})\frac{a+kx}{a+k}|=\|g_1\|=1$.
 However, the natural extension $\widetilde{\phi}$ is not positive.
  Choose $$f(x)=\begin{cases}
                0 & x\in[0,x_{0}] \\
                \frac{k}{1-x_{0}}(x-1)+k & x\in(x_{0},1] ,\\
                 \end{cases}$$
 then $f(x)\geq0$, consider its direct sum decomposition $f=\lambda+g$, where $\lambda=\frac{(a+k)f(0)-af(1)}{k}=-a$ and
 $$g(x)=\begin{cases}
              a & x\in[0,x_{0}] \\
              \frac{k}{1-x_{0}}(x-1)+a+k & x\in(x_{0},1].\\
             \end{cases}$$
 Then  $\widetilde{\phi}(f)(x)=\lambda+\phi(g)=-a+a\frac{a+kx}{a+k}=\frac{ak(x-1)}{a+k}\leq0$.
\end{remark}

Although the constant function 1 is not in $C[0,1]_{(a,k)}$, we still can approximate it using special functions in the subspace, which are referred as test functions. Moreover, the behavior of a positive
linear map on the test functions indicate the positivity of its extension.

\begin{definition}
For each $1\geq\delta>0$, a lower test function $e_{\delta}\in C[0,1]_{(a,k)}$ is a continuous function which has value 1 on $[\delta,1]$ and has value $\frac{a}{k+a}
$ at 0. Denote by $L$ the set of all such lower test functions.
For each  $1>\sigma\geq0$, a upper test function $\gamma_{\sigma}\in C[0,1]_{(a,k)}$ is a continuous function which has value 1 on $[0,1-\sigma]$ and has value $\frac{a+k}{a}$ at 1. Denote by $S$ the set of all such upper test functions.
 \end{definition}

\begin{prop}
Let  there be given a positive linear map $\phi$ from $C[0,1]_{(a,k)}$ to $C[0,1]$, then $\phi$ is positively extendible if and only if the following estimation holds: $\inf\limits_{\gamma\in S}\{\phi(\gamma)\}\geq1\geq \sup\limits_{e\in L}\{\phi(e)\}$ .
\end{prop}
\begin{proof} First suppose $\phi$ is positively extendible, then  $\widetilde{\phi}$ is positive. For any $\gamma\in S, e\in L$, $\gamma>1>e$,$\widetilde{\phi}(\gamma)\geq1\geq\widetilde{\phi}(e)$, then $\phi(\gamma)\geq1\geq\phi(e)$, and $\inf\limits_{\gamma\in S}\{\phi(\gamma)\}\geq1\geq \sup\limits_{e\in L}\{\phi(e)\}$.

Conversely, suppose  $\inf\limits_{\gamma\in S}\{\phi(\gamma)\}\geq1\geq \sup\limits_{e\in L}\{\phi(e)\}$, we need to show $\widetilde{\phi}: C[0,1]\rightarrow C[0,1]$ defined by $\widetilde{\phi}(f)=\lambda+\phi(g)$ is positive. For $f\in C[0,1]$ and $f\geq0$ with the decomposition  $f=\lambda+g $, then we need to show $\widetilde{\phi}(f)=\lambda+\phi(g)\geq0$. Thus we need to prove $\phi(g)\geq-\lambda$ if $g\geq-\lambda$.

$\mathit{Case I:}$ If $\lambda=0$, then $f\in C[0,1]_{(a,k)}$, we have $\widetilde{\phi}(f)=\phi(f)\geq0 $ since $f\geq0$ and $\phi$ is positive.

$\mathit{Case II:}$ If $\lambda<0$, then $-\lambda>0$, then we can find a $\gamma_{\sigma}\in S$ such that $(-\lambda)\gamma_{\sigma}\leq g$. Since $\phi$ is positive, one has that $\phi(g)\geq \phi(-\lambda\gamma_{\sigma})=-\lambda\phi(\gamma_{\sigma})\geq -\lambda$.

$\mathit{Case III:}$ If $\lambda>0$, then $-\lambda<0$, then we can find a $e_{\delta}\in L$ such that $g\geq -\lambda e_{\delta}$. Since $\sup\limits_{0<\delta\leq1}\phi(e_{\delta})\leq1$, one has $-\lambda\phi(e_{\delta})\geq -\lambda$. Therefore, $\phi(g)\geq-\lambda\phi(e_{\delta})\geq-\lambda$.

Hence $\widetilde{\phi}$ is positive and $\phi$ is positively extendible.
\end{proof}

\section{approximation results on $[0,1]$}

Given a Markov operator between $\mathrm{C}[0,1]$ which keeps subspace $\mathrm{C}[0,1]_{(a,k)}$, we want to approximate it by an average of homomorphisms on $\mathrm{C}[0,1]$ and additionally require that the average also keeps the same subspace. Since the subhomogeneity in consideration arise at 0 and 1, we need to investigate the measures induced by point evaluations of a Markov operator at endpoints 0 and 1.
\begin{lem}\label{point mass}
 Let there be given a unital positive linear map $\phi$ on $C[0,1]$ which preserves $C[0,1]_{(a,k)}$. Then the measures induced by evaluations of $\phi$ at 0 and 1 actually concentrate on 0 and 1 respectively, in other words, $\phi(f)(0)=f(0)$ and $\phi(f)(1)=f(1)$ for all $f\in C[0,1]$.
 \end{lem}
\begin{proof} For any fixed $y\in[0,1]$, $f\rightarrow \phi(f)(y)$ gives a positive Borel probability measure on $[0,1]$, say $\mu_{y}$, then $\phi(f)(y)=\int_{0}^{1}fd\mu_{y}$.

Then for all $g\in C[0,1]_{(a,k)}$, $\phi(g)(0)=\int_{0}^{1}gd\mu_{0}$ and $\phi(g)(1)=\int_{0}^{1}gd\mu_{1}$, since $\phi(g)(0)=\frac{a}{a+k}\phi(g)(1)$, one has that $$\int_{0}^{1}gd\mu_{0}=\frac{a}{a+k}\int_{0}^{1}gd\mu_{1}.$$

For any $\delta>0$, choose a finite $\delta$-dense subset $\{x_{1},x_{2},...,x_{n}\}\subset[0,1]$ with $x_{1}=0, x_{n}=1$. Then for every $x\in[0,1]$, there is a $x_{i}$ in the finite subset above such that $dist(x,x_{i})<\delta$. Given a partition of $[0,1]$, denoted by $\{X_{1},X_{2},...,X_{n}\}$, with each $X_{i}$ being a connected Borel set, satisfying the following conditions:
$$\begin{aligned}
&(1) \,x_{i}\in X_{i}, i=1,2,...,n;\\
&(2)\, [0,1]=\bigcup_{i=1}^n X_{i}, X_{i}\bigcap X_{j}=\emptyset\;\text{for}\;i\neq j;\\
&(3) \,dist(x,x_{i})<\delta \;\text{if} \; x\in X_{i}.\\
\end{aligned}$$
For any fixed $\delta$ above, there exists a $\delta_{0}>0$ such that $[0,\delta_{0}]\subseteq X_{1}$. Choose a function $e_{\delta}\in C[0,1]_{(a,k)}$ as follows: $$e_{\delta_{0}}=\begin{cases}
                  \frac{a}{a+k}+\frac{1-\frac{a}{a+k}}{\delta_{0}}x & x\in[0,\delta_{0}] \\
                   1 & x\in(\delta_{0},1].\\
                 \end{cases}
               $$
Then
  \[\phi(e_{\delta_{0}})(0)=\int_{0}^{1}e_{\delta_{0}}d\mu_{0}=\int_{X_{1}}e_{\delta_{0}}d\mu_{0}+\mu_{0}(X_{2})+\mu_{0}(X_{3)}+...+\mu_{0}(X_{n})\]
 \[=\int_{X_{1}}e_{\delta_{0}}d\mu_{0}+1-\mu_{0}(X_{1)}=1+\int_{X_{1}}(e_{\delta_{0}}-1)d\mu_{0},\]
\[ \phi(e_{\delta_{0}})(1)=\int_{0}^{1}e_{\delta_{0}}d\mu_{1}=\int_{X_{1}}e_{\delta_{0}}d\mu_{1}+\mu_{1}(X_{2})+\mu_{1}(X_{3)}+...+\mu_{1}(X_{n})\]
 \[=\int_{X_{1}}e_{\delta_{0}}d\mu_{1}+1-\mu_{1}(X_{1)}=1+\int_{X_{1}}(e_{\delta_{0}}-1)d\mu_{1}.\]
Since \[ \phi(e_{\delta_{0}})(0)=\frac{a}{a+k}\phi(e_{\delta_{0}})(1),\] one gets that \[1+\int_{X_{1}}(e_{\delta_{0}}-1)d\mu_{0}=\frac{a}{a+k}(1+\int_{X_{1}}(e_{\delta_{0}}-1)d\mu_{1}).\] Then
\begin{equation*}\label{sum}
\frac{k}{a+k}+\int_{X_{1}}(e_{\delta_{0}}-1)d\mu_{0}=\frac{a}{a+k}\int_{X_{1}}(e_{\delta_{0}}-1)d\mu_{1}.
\end{equation*}
 Since $\frac{a}{a+k}\leq e_{\delta_{0}}\leq1$, and $\mu_0$ and $\mu_1$ are positive measures, one has that $-\frac{k}{a+k}\leq\int_{X_{1}}(e_{\delta_{0}}-1)d\mu_{0}\leq0$ and $-\frac{k}{a+k}\leq\int_{X_{1}}(e_{\delta_{0}}-1)d\mu_{1}\leq0$. Therefore, $$0\leq\frac{k}{a+k}+\int_{X_{1}}(e_{\delta_{0}}-1)d\mu_{0}\leq\frac{k}{a+k}.$$ Then the left hand side of equation above is $\geq0$ and the right hand side is $\leq0$. Hence $$\frac{k}{a+k}+\int_{X_{1}}(e_{\delta_{0}}-1)d\mu_{0}=\frac{a}{a+k}\int_{X_{1}}(e_{\delta_{0}}-1)d\mu_{1}=0.$$

Then we have $\int_{X_{1}}(1-e_{\delta_{0}})d\mu_{0}=\frac{k}{a+k}$, and $0\leq\int_{X_{1}}(1-e_{\delta_{0}})d\mu_{0}\leq\frac{k}{a+k}\mu_{0}(X_{1})$. Hence $\mu_{0}(X_{1})\geq1$, so $\mu_{0}(X_{1})=1.$ Since $\delta$ is arbitrary, then $\mu_0(\{0\})=1$.

In a similar way, for any fixed $\delta$, there exists a $\delta_{1}$ such that $[\delta_{1},1]\subseteq X_{n}$. One can choose a function $\gamma_{\delta_1}\in C[0,1]_{(a,k)}$ as follows:
\[\gamma_{\delta_{1}}=\begin{cases}
                 1 & x\in[0,\delta_{1}) \\
                   1+\frac{\frac{a+k}{a}-1}{1-\delta_{1}}(x-\delta_{1}) & x\in(\delta_{1},1].\\
                 \end{cases}
               \]
Then
\[\phi(\gamma_{\delta_{1}})(0)=\int_{0}^{1}\gamma_{\delta_{1}}d\mu_{0}=\int_{X_{n}}\gamma_{\delta_{1}}d\mu_{0}+\mu_{0}(X_{1})+\mu_{0}(X_{2)}+...+\mu_{0}(X_{n-1})\]
\[=1-\mu_{0}(X_{n})+\int_{X_{n}}\gamma_{\delta_{1}}d\mu_{0}=1+\int_{X_{n}}(\gamma_{\delta_{1}}-1)d\mu_{0},\]
\[\phi(\gamma_{\delta_{1}})(1)=\int_{0}^{1}\gamma_{\delta_{1}}d\mu_{1}=\int_{X_{n}}\gamma_{\delta_{1}}d\mu_{1}+\mu_{1}(X_{1})+\mu_{1}(X_{2)}+...+\mu_{1}(X_{n-1})\]
\[=1-\mu_{1}(X_{n})+\int_{X_{n}}\gamma_{\delta_{1}}d\mu_{1}=1+\int_{X_{n}}(\gamma_{\delta_{1}}-1)d\mu_{1}.\]
Since \[ \frac{a+k}{a}\phi(\gamma_{\delta_{1}})(0)=\phi(\gamma_{\delta_{1}})(1),\]
one has that \[\frac{a+k}{a}(1+\int_{X_{n}}(\gamma_{\delta_{1}}-1)d\mu_{0})=1+\int_{X_{n}}(\gamma_{\delta_{1}}-1)d\mu_{1},\]
\begin{equation*}\label{sum}
\frac{k}{a}+\frac{a+k}{a}\int_{X_{n}}(\gamma_{\delta_{1}}-1)d\mu_{0}=\int_{X_{n}}(\gamma_{\delta_{1}}-1)d\mu_{1}.
\end{equation*}
While $0\leq\gamma_{\delta_{1}}-1\leq\frac{k}{a}$, and $0\leq\int_{X_{n}}(\gamma_{\delta_{1}}-1)d\mu_{0}\leq\frac{k}{a},\;0\leq\int_{X_{n}}(\gamma_{\delta_{1}}-1)d\mu_{1}\leq\frac{k}{a}$. Then the left hand side of the equation above is $\geq\frac{k}{a}$ and the right hand side is $\leq\frac{k}{a}$.  Hence $$\frac{k}{a}+\frac{a+k}{a}\int_{X_{n}}(\gamma_{\delta_{1}}-1)d\mu_{0}=\int_{X_{n}}(\gamma_{\delta_{1}}-1)d\mu_{1}=\frac{k}{a}.$$
Then we have $$\frac{k}{a}\leq\int_{X_{n}}(\gamma_{\delta_{1}}-1)d\mu_{1}\leq\frac{k}{a}\mu_{1}(X_{n}),$$ then $\mu_{1}(X_{n})\geq1$, so $\mu_{1}(X_{n})=1$. Since $\delta$ is arbitrary, then $\mu_1(\{1\})=1$.

Hence $\phi(f)(0)=f(0)$ and $\phi(f)(1)=f(1)$ for all $f\in C[0,1]$.
\end{proof}
\begin{cor}\label{lem:cont1}
Let $\phi: C[0,1]_{(a,k)}\rightarrow C[0,1]_{(a,k)}$ be defined as $\phi(f)=f\circ\lambda$ for some continuous $\lambda:[0,1]\rightarrow[0,1]$, then $\lambda(0)=0$ and $\lambda(1)=1$.
\end{cor}
\begin{proof} Choose some injective function $f$ and apply the lemma above.
\end{proof}

\textbf{Next we proceed to prove Theorem \ref{thm1}:}

\begin{proof} The proof is inspired by Li's proof in \cite{Li}, the thing is we need more accurate analysis for the endpoints. We spell it out in full detail for the convenience of readers.

$\mathit{Step I}.$ For all  $f\in F$ and $y\in [0,1]$, we approximate $\phi(f)(y)$ by a finite sum of evaluations of $f$ with continuous coefficients.

For any $\epsilon>0$, there is a $\delta_{0}$ such that for any $x_{1},x_{2}\in[0,1]$, if $dist(x_{1},x_{2})<\delta_{0}$, then $$|f(x_{1})-f(x_{2})|<\frac{\epsilon}{4}$$ for all $f\in F.$
Choose a finite subset $\{x_{1},x_{2},...,x_{n}\}\subset[0,1]$ which is $\delta_{0}$-dense in $[0,1]$ with $x_{1}=0,x_{n}=1$. Then for every $x\in[0,1]$, there is a $x_{i}$ in the finite subset such that $dist(x,x_{i})<\delta_{0}$. Choose a partition of $[0,1]$, denoted by $\{X_{1},X_{2},...,X_{n}\}$, with each $X_{i}$ being a connected Borel set, satisfying the following conditions:
$$\begin{aligned}
&(1)\, x_{i}\in X_{i}, i=1,2,...,n;\\
&(2)\, [0,1]=\bigcup_{i=1}^n X_{i},X_{i}\bigcap X_{j}=\emptyset \;\text{for}\; i\neq j;\\
&(3)\, dist(x,x_{i})<\delta_{0} \;\text{if}\; x\in X_{i}.\\
\end{aligned}$$
Then for any probability measure $\mu$ on $[0,1]$, there are non-negative numbers $\lambda_{1},\lambda_{2},...,\lambda_{n}$ with $\sum_{i=1}^{n}\lambda_{i}=1$ such that
\[|\mu(f)-\sum_{i=1}^{n}\lambda_{i}f(x_{i})|<\frac{\varepsilon}{4}\;\text{for all}\;f\in F.\]
Actually, we have\[|\mu(f)-\sum_{i=1}^{n}\mu(X_{i})f(x_{i})|<|\sum_{i=1}^{n}\int_{X_{i}}(f(x)-f(x_{i}))d\mu|\leq\sum_{i=1}^{n}\frac{\epsilon}{4}\mu(X_{i})=\frac{\epsilon}{4}.\]
So one may choose $\lambda_{i}=\mu(X_{i})$.

For any fixed  $y\in [0,1],\;f\rightarrow\phi(f)(y)$ is a probability measure on [0,1], thus from above, there are non-negative numbers $\lambda_{1y},\lambda_{2y},...,\lambda_{ny}$ with $\sum_{i=1}^{n}\lambda_{iy}=1$ such that \[|\phi(f)(y)-\sum_{i=1}^{n}\lambda_{iy}f(x_{i})|<\frac{\epsilon}{4}\;\text{for all}\;f\in F.\]
By continuity of $\phi(f)$, this estimation holds in a neighborhood of y. Since $[0,1]$ is compact, we can find an open cover $\{V_{j}: j=1,2,...,R\}$ of $[0,1]$, such that
$$\begin{aligned}
&(1) 0\in V_{1}, 0\notin \bigcup^{R}_{j=2}V_{j},1\in V_{R},1\notin\bigcup^{R-1}_{j=1} V_{j},\\
&(2) y_{j}\in V_{j}, y_{1}=0,y_{R}=1, j=2,...,R-1.\\
\end{aligned}$$
Then one has that

\[|\phi(f)(y)-\sum_{i=1}^{n}\lambda_{iy_{j}}f(x_{i})|<\frac{\epsilon}{4}\;\text{for all}\;y\in V_{j}\;and\;f\in F.\]
Let $\{h_{j}\}_{1}^{R}$ be a partition of unity subordinate to $\{V_{j}\}_{1}^{R}$. Define $\lambda_{i}(y)=\sum_{j=1}^{R}\lambda_{iy_{j}}h_{j}(y)$, then $\lambda_{i}\in C[0,1], \lambda_{i}(0)=\lambda_{i0}=\mu_{0}(X_{i}), \lambda_{i}(1)=\lambda_{i1}=\mu_{1}(X_{i})$ and

\[\sum_{i=1}^{n}\lambda_{i}(y)=\sum_{i=1}^{n}(\sum_{j=1}^{R}\lambda_{iy_{j}}h_{j}(y))=\sum_{j=1}^{R}h_{j}(y)=1.\]
Hence \[|\phi(f)(y)-\sum_{i=1}^{n}\lambda_{i}(y)f(x_{i})|<\frac{\epsilon}{4}\] for all $y\in [0,1]$ and $f\in F$.

$\mathit{Step II}.$ We approximate the finite sum of evaluations above by a linear map $w$ on $C[0,1]$ defined as an integral of the composition with some continuous function $h(y,t)$ from $[0,1]\times[0,1]$ to $[0,1]$.

Let there be given a $\delta>0$ to be used later with $4n\delta\sup_{f\in F}\|f\|<\frac{\epsilon}{4}$.

First, we define continuous maps $G_{0},G_{1},...G_{n}:[0,1]\rightarrow[0,1]$ by
\[ G_{0}(y)=0,G_{j}(y)=\sum_{i=1}^j\lambda_{i}(y),\;j=1,2,...,n.\] For each $y\in [0,1]$, these points $\{G_i(y)\}_i$ give to a partition of $[0,1]$. Moreover, for each $j$, we define
\[f_{j}(y)=\min\{G_{j-1}(y)+\delta;\frac{G_{j-1}(y)+G_{j}(y)}{2}\},\] and
\[g_{j}(y)=\max\{G_{j}(y)-\delta;\frac{G_{j-1}(y)+G_{j}(y)}{2}\}.\]

To define $h(y,t)$, we only need to define $h(y,t)$ on each $[0,1]\times[G_{j-1}(y),G_{j}(y)]$, let us denote by $h_{j}(y,t)$ this restriction.
For our purpose, we choose the following $h_j(y,t)$:
\[h_{j}(y,t)=\begin{cases}
                  \frac{x_{j}(t-G_{j-1}(y))}{\delta} & t\in [G_{j-1}(y),f_{j}(y)] \\
                  \min(x_{j},\frac{x_{j}(G_{j}(y)-G_{j-1}(y))}{2\delta}) & t\in [f_{j}(y),g_{j}(y)] \\
                  \frac{x_{j}(G_{j}(y)-t)}{\delta} & t\in [g_{j}(y),G_{j}(y)]. \\
                 \end{cases}
               \]
Then $h_{j}(y,t)$ satisfies that for any $y\in [0,1],$
\[|h_{j}(y,t_{1})-h_{j}(y,t_{2})|\leq\frac{x_{j}|t_{1}-t_{2}|}{\delta},\]
and $h_{j}(y,t):[0,1]\times[G_{j-1}(y),G_{j}(y)]\rightarrow[0,1]$ is continuous.
Hence $h(y,t):[0,1]\times[0,1]\rightarrow[0,1]$ is continuous and
\[|h(y,t_{1})-h(y,t_{2})|\leq\frac{|t_{1}-t_{2}|}{\delta}.\]

Define $w:C[0,1]\rightarrow C[0,1]$ by $w(f)(y)=\int_{0}^{1}f(h(y,t))dt$ for $f\in C[0,1], y\in[0,1]$.

Then for all $f\in F$, one has that
\[
\begin{aligned}
&|\sum_{i=1}^n\lambda_{i}(y)f(x_{i})-w(f)(y)|\\
=& |\sum_{i=1}^n\int_{G_{i-1}(y)}^{G_{i}(y)}f(x_{i})dt-\int_{0}^{1}f(h(y,t))dt|\\
\leq&\sum_{i=1}^n|\int_{G_{i-1}(y)}^{G_{i}(y)}f(x_{i})-f(h(y,t))dt|\\
=&\sum_{i=1}^n|(\int_{G_{i-1}(y)}^{f_{i}(y)}+\int_{f_{i}(y)}^{g_{i}(y)}+\int_{g_{i}(y)}^{G_{i}(y)})(f(x_{i})-f(h(y,t)))dt|\\
\leq&\sum_{i=1}^n(2\delta\sup_{f\in F}\|f\|+0+2\delta\sup_{f\in F}\|f\|)\\
=&4n\delta\sup_{f\in F}\|f\|<\frac{\epsilon}{4}.
\end{aligned}
\]
$\mathit{Step III}.$ Finally we shall choose $N$ continuous maps on $[0,1]$ to define the homomorphisms. Such maps come from $h(y,t)$ by specifying $N$ values of $t$.
Moreover, we choose these maps in a way that the corresponding homomorphisms preserve the subspace $C[0,1]_{(a,k)}$.

Choose an integer $N_{1}>0$ with $\frac{1}{N_{1}}<\delta\delta_{0}$, and choose specified values of $t$ as $t_{j}=\frac{j}{N_{1}}\in[0,1], j=1,2,...,N_{1}$. Then the linear map $w$ can be approximated by the average of the homomorphisms induced by $h(y,t_j), j=1,...,N_1$. 

This is shown as follows:
set $$w(f)(y)=\sum_{j=1}^{N_{1}}\int_{t_{j-1}}^{t_{j}}f(h(y,t))dt,$$
then 
\[
\begin{aligned}
&|w(f)(y)-\frac{1}{N_{1}}\sum_{j=1}^{N_{1}}f(h(y,t_{j}))|\\
=&| \sum_{j=1}^{N_{1}}\int_{t_{j-1}}^{t_{j}}f(h(y,t))-f(h(y,t_{j}))dt|\\
<& \sum_{j=1}^{N_{1}}\int_{t_{j-1}}^{t_{j}}\frac{\epsilon}{4}=\frac{\epsilon}{4}.
\end{aligned}
\]
for all $f\in F$, where $|f(h(y,t))-f(h(y,t_{j}))|<\frac{\epsilon}{4}$, since $|h(y,t)-h(y,t_{j})|<\delta_{0}$ (Note that $|t-t_{j}|\leq\frac{1}{N_{1}}<\delta\delta_{0}$).

However, the average of these homomorphisms may not preserve the subspace. To fix this problem, we have to make more delicate choices. 

By Lemma \ref{point mass}, we know that for $\phi(f)(0)$, the coefficients are: $$\lambda_1(0)=1, \lambda_2(0)=...=\lambda_n(0)=0.$$ Similarly, for $\phi(f)(1)$, one has that $$\lambda_1(1)= \lambda_2(1)=...=\lambda_{n-1}(1)=0, \lambda_n(1)=1.$$ Therefore, one sets that $h(0,t)=0$; and $$h(1,t)=\begin{cases}
                   \frac{t}{\delta} & t\in [0,\delta] \\
                   1 & t\in [\delta,1-\delta] \\
                  \frac{1-t}{\delta} & t\in[1-\delta,1] .\\
                 \end{cases}$$
If we can choose new $t_j$ such that $h(0,t_j)=0$ and $h(1,t_j)=1$, then the corresponding homomorphisms will fit our purpose. Choose those $j$ such that $\delta\leq\frac{j}{N_1}\leq1-\delta$, i.e., $\delta N_1\leq j\leq (1-\delta)N_1$. Denote by $N$ the number of such $j$, then $N=\lfloor(1-\delta)N_1\rfloor-\lceil\delta N_1\rceil+1$.

We are going to show that the average of these $N$ homomorphisms can approximate the average of the original $N_1$ homomorphisms:
$$\begin{aligned}
&|\frac{1}{N_{1}}\sum_{j=1}^{N_{1}}f(h(y,t_{j}))-\frac{1}{N}\sum_{j=\lceil\delta N_1\rceil}^{\lfloor(1-\delta)N_1\rfloor}f(h(y,t_{j}))|\\
=&\frac{1}{N_{1}}|\sum_{j=1}^{N_{1}}f(h(y,t_{j}))-\frac{N_1}{N}\sum_{j=\lceil\delta N_1\rceil}^{\lfloor(1-\delta)N_1\rfloor}f(h(y,t_j))|\\
=&\frac{1}{N_{1}}|\sum_{j=1}^{\lceil N_{1}\delta\rceil-1}f(h(y,t_{j}))+\sum_{\lfloor N_{1}(1-\delta)\rfloor+1}^{N_{1}}f(h(y,t_{j}))+(1-\frac{N_1}{N})\sum_{j=\lceil N_{1}\delta\rceil}^{\lfloor N_{1}(1-\delta)\rfloor}f(h(y,t_{j}))|\\
\leq&\frac{1}{N_{1}}(N_{1}\delta \sup_{f\in F}\|f\|+(N_{1}\delta+1) \sup_{f\in F}\|f\|+(2N_{1}\delta+1) \sup_{f\in F}\|f\|)\\
\leq&5\delta \sup_{f\in F}\|f\|\leq\frac{\epsilon}{4}.\\
\end{aligned}$$
(Note that the above estimation holds since
$$\begin{aligned}
&N_1-N=N_1-(\lfloor(1-\delta)N_1\rfloor-\lceil\delta N_1\rceil+1)\\
=&N_1-1-(\lfloor(1-\delta)N_1\rfloor-\lceil\delta N_1\rceil)\\
\leq &N_1-1-[(1-\delta)N_1-1-(\delta N_1+1)]\\
=&2\delta N_1+1.)\\
\end{aligned}$$
For those new $j$, let us define $\phi_{j}:C[0,1]\rightarrow C[0,1]$ by $\phi_{j}(f(y))=f(h(y,t_{j}))$. Then
$$\begin{aligned}
&|\phi(f)(y)-\frac{1}{N}\sum_{j=\lceil\delta N_1\rceil}^{\lfloor(1-\delta)N_1\rfloor}\phi_{j}(f)(y)|\\
=&|\phi(f)(y)-\frac{1}{N}\sum_{j=\lceil\delta N_1\rceil}^{\lfloor(1-\delta)N_1\rfloor}f(h(y,t_{j}))|\\
\leq&|\phi(f)(y)-\sum_{i=1}^{n}\lambda_{i}(y)f(x_{i})|+|\sum_{i=1}^n\lambda_{i}(y)f(x_{i})-w(f)(y)|\\
+&|w(f)(y)-\frac{1}{N_{1}}\sum_{j=1}^{N_{1}}f(h(y,t_{j}))|+|\frac{1}{N_{1}}\sum_{j=1}^{N_{1}}f(h(y,t_{j}))-\frac{1}{N}\sum_{j=\lceil\delta N_1\rceil}^{\lfloor(1-\delta)N_1\rfloor}f(h(y,t_{j}))|\\
\leq&\frac{\epsilon}{4}+\frac{\epsilon}{4}+\frac{\epsilon}{4}+\frac{\epsilon}{4}\leq\epsilon.\\
\end{aligned}$$
\end{proof}

\begin{remark} From the proof above, one can see that for any integer $M_1\geq N_1$, there are an corresponding integer $M$ and $M$ homomorphisms such that the average of these $M$ morphisms also meets the requirements.

Hence, under the assumption of Theorem \ref{thm1}, there is a sequence of positive integers $\{L_j\}_{j=1}^{\infty}$ with large enough lower bound, and there are corresponding $L_j$ morphisms for each $j$, such that the average of these $L_j$ morphisms meets the requirements.

\end{remark}

Consider the C*-algebras $$A=\{f\in C([0,1], M_n)\mid f(0)=\text{diag}(d\otimes id_{a}, 0\otimes id_{k}) , f(1)=d\otimes id_{a+k}\}$$ and $$B=\{f\in C([0,1], M_m)\mid f(0)=\text{diag}(e\otimes id_{a}, 0\otimes id_{k}) , f(1)=e\otimes id_{a+k}\},$$ where $d$ and $e$ are matrices of appropriate sizes. Theorem \ref{thm1} can be used to build up $*$-homomorphisms between certain C*-algebras.

\begin{cor} Consider the C*-algebras $A$ as above, for any $\epsilon>0$, and any finite subset $F\subseteq AffTA=C[0,1]_{(a,k)}$, there is an integer $N>0$, such that for any C*-algebra $B$ of the form above with generic fibre size being $N$ multiples of generic fibre size of $A$, and a unital positive linear map $\xi$ on $C[0,1]$ which preserves $C[0,1]_{(a,k)}$, then there is a $*$-homomorphism $\phi$ from $A$ to $B$, such that $$\|\xi(f)-AffT\phi(f)\|<\epsilon$$ for all $f\in F$.
\end{cor}
\begin{proof}
We take $N$ as in Theorem \ref{thm1}, and the corresponding $N$ continuous maps $h(y, t_1),..., h(y, t_N)$ on $[0,1]$. By the constructions of $h(y, t_j)$, there are unitary matrices $U_0$ and $U_1$, such that for each $g\in A$, $$U_0\text{diag}(g(h(0, t_1)),...,g(h(0, t_N)))U_0^*=\text{diag}(e\otimes id_{a}, 0\otimes id_{k}),$$ and $$U_1\text{diag}(g(h(1, t_1)),...,g(h(1, t_N)))U_1^*=e\otimes id_{a+k},$$ for some matrix $e$. By choosing a continuous path of unitaries $U(t)$ connecting $U_0$ and $U_1$, one can define $\phi: A\rightarrow B$ as $$\phi(g)(y)=U(y)\text{diag}(g\circ h(y, t_1),...,g\circ h(y, t_N))U^*(y).$$ Keeping mind the correspondence $f=(tr\otimes \delta_t)(g)$(see Proposition 2.1 in \cite{Rak}), and applying Theorem \ref{thm1}, one has that $$\|\xi(f)-AffT\phi(f)\|<\epsilon$$ for all $f\in F$.

\end{proof}

Now we consider the case involving different subspaces.

\begin{lem}\label{small} Let $\mu$ be any Borel probability measure on [0,1], then for any $x\in[0,1]$ and $\epsilon>0$, there is a $\delta>0$, such that $\mu(B^{0}(x,\delta))<\epsilon$ (where $B^{0}(x,\delta)$ denotes the neighborhood of $x$ excluding the center).
\end{lem}
\begin{proof}Let $D_{k}=(B(x,\frac{1}{k})\setminus B(x,\frac{1}{1+k}))\cap[0,1]$ and $S_{n}=\sum_{k=1}^{n}\mu(D_{k})$. Then $S_{n}$ is increasing and upper bounded, so $\{S_{n}\}$ converges. Then for any $\epsilon>0$, there exists $N>0$ such that $\sum_{k=N}^{\infty}\mu(D_{k})<\epsilon$. Hence there exists $\delta=\frac{1}{N}$, such that $\mu(B^{0}(x,\delta))<\epsilon$, since $B^{0}(x,\delta)=\bigcup\limits_{k=N}^{\infty}D_k$.
The proof is complete.
\end{proof}
Examples in Section 2 show that the measures induced by evaluations of a Markov operator at 0 and 1 actually could involve as many points as you want, so we investigate the behavior of induced measures with respect to a given partition of $[0,1]$ coming from an approximation.

\begin{lem}\label{mea2}
Given any unital positive linear map $\phi$ from $C[0,1]$ to $C[0,1]$ which sends $C[0,1]_{(a,k)}$ to $C[0,1]_{(b,k)}$, denote by $\mu_0$ and $\mu_1$ the measures induced by evaluations of $\phi$ at 0 and 1. Let there be given a partition $\{X_{1},X_{2},...,X_{n}\}$ of $[0,1]$, where $X_{i}$ is a connected Borel set and $0\in X_{1}$, $1\in X_{n}$.

Then we have the following distribution of $\mu_0$ and $\mu_1$ with respect to the partition:
\begin{align}
&\mu_{0}(X_{i})=\frac{b}{b+k}\mu_{1}(X_{i}), i=2,...,n-1, \\
&\frac{a}{a+k}\mu_{0}(X_{1})+\mu_{0}(X_{n})=\frac{b}{b+k}[\frac{a}{a+k}\mu_{1}(X_{1})+\mu_{1}(X_{n})].
\end{align}
\end{lem}
\begin{proof} First we prove the first relation. This is shown as follows. Choose a continuous function which is almost supported on $X_i$, then apply $\phi$, and compare the  evaluations at 0 and 1, one can get the relation.

For $i=2,...,n-1$, let us assume $a_{i}=\sup\{x: x\in X_{i}\}$, $b_{i}=\inf\{x: x\in X_{i}\}$. We deal with the problem in several cases, and we only spell it out in one case, similar proof works for other cases.

If $a_{i},b_{i}\in X_{i}$, that is $X_{i}=[b_{i}, a_{i}]$. For $\forall \epsilon>0$, by Lemma \ref{small}, there exists a $\delta>0$ such that $\mu_{0}((b_{i}-\delta,b_{i}))<\epsilon, \mu_{1}((b_{i}-\delta,b_{i}))<\epsilon, \mu_{0}((a_{i},a_{i}+\delta))<\epsilon, \mu_{1}((a_{i},a_{i}+\delta))<\epsilon$. Choose a function $g_{i}(x)$ as follows: \[ g_{i}(x)=\begin{cases}
                 0 & x\in(b_{i}-\delta,a_{i}+\delta)^{c} \\
                   \frac{x-b_{i}}{\delta}+1 & x\in(b_{i}-\delta,b_{i}) \\
                   -\frac{x-a_{i}}{\delta}+1 & x\in(a_{i},a_{i}+\delta) \\
                     1 & x\in[b_{i},a_{i}],\\
                 \end{cases}
               \] then \[\phi(g_{i})(0)=\int_{[0,1]}g_{i}d\mu_{0}=\mu_{0}(X_{i})+\int_{(b_{i}-\delta,b_{i})}g_{i}d\mu_{0}+\int_{(a_{i},a_{i}+\delta)}g_{i}d\mu_{0},\] and \[\phi(g_{i})(1)=\int_{[0,1]}g_{i}d\mu_{1}=\mu_{1}(X_{i})+\int_{(b_{i}-\delta,b_{i})}g_{i}d\mu_{1}+\int_{(a_{i},a_{i}+\delta)}g_{i}d\mu_{1}.\]
Let us take $$\epsilon_{i0}=\int_{(b_{i}-\delta,b_{i})}g_{i}d\mu_{0}+\int_{(a_{i},a_{i}+\delta)}g_{i}d\mu_{0},$$ and $$
\epsilon_{i1}=\int_{(b_{i}-\delta,b_{i})}g_{i}d\mu_{1}+\int_{(a_{i},a_{i}+\delta)}g_{i}d\mu_{1},$$ then $\epsilon_{i0}<2\epsilon, \epsilon_{i1}<2\epsilon$ and \[\frac{\mu_{0}(X_{i})+\epsilon_{i0}}{\mu_{1}(X_{i})+\epsilon_{i1}}=\frac{b}{b+k} \] (since $\phi(g_{i})\in C[0,1]_{(b,k)}$).

Hence $$|\mu_{0}(X_{i})-\frac{b}{b+k}\mu_{1}(X_{i})|<4\epsilon,$$ since $\epsilon$ is arbitrary, one gets that $$\mu_{0}(X_{i})=\dfrac{b}{b+k}\mu_{1}(X_{i}), i=2,...,n-1.$$ 

Next we prove the second relation. We use similar ideas, i.e., we choose a function which is almost supported around 0 and 1, apply $\phi$, and then compare the evaluations.

For $i=1$ and $i=n$, we know that $0=inf\{x:x\in X_{1}\}$, $1=sup\{x: x\in X_{n}\}$, and suppose $a_{1}=\sup\{x: x\in X_{1}\}$, $b_{n}=\inf\{x: x\in X_{n}\}$.
Suppose $a_{1}\in X_{1},b_{n}\in X_{n}$, that is $X_{1}=[0,a_{1}]$, $X_{n}=[b_{n},1]$. For $\forall\epsilon>0$, by Lemma \ref{small}, there exists a $\delta>0$ such that 
$$\begin{aligned}
&\mu_{0}((b_{n}-\delta,b_{n}))<\epsilon,\mu_{1}((b_{n}-\delta,b_{n}))<\epsilon,\\
&\mu_{0}((a_{1},a_{1}+\delta))<\epsilon,\mu_{1}((a_{1},a_{1}+\delta))<\epsilon.\\
\end{aligned}$$
We choose a function $g(x)$ as follows: \[ g(x)=\begin{cases}
                 \frac{a}{k+a} & x\in[0,a_{1}] \\
                   -\frac{a(x-a_{1})}{\delta(a+k)}+\frac{a}{a+k} & x\in(a_{1},a_{1}+\delta) \\
                   0 & x\in([a_{1}+\delta,b_{n}-\delta] \\
                   \frac{x-b_{n}}{\delta}+1 & x\in(b_{n}-\delta,b_{n}) \\
                     1 & x\in[b_{n},1] \\
                 \end{cases}
               \]
Then \[\phi(g)(0)=\int_{[0,1]}gd\mu_{0}=\mu_{0}(X_{n})+\frac{a}{a+k}\mu_{0}(X_{1})+\int_{(a_{1},a_{1}+\delta)}gd\mu_{0}+\int_{(b_{n}-\delta,b_{n})}gd\mu_{0},\] and
\[\phi(g)(1)=\int_{[0,1]}gd\mu_{1}=\mu_{1}(X_{n})+\frac{a}{a+k}\mu_{1}(X_{1})+\int_{(a_{1},a_{1}+\delta)}gd\mu_{1}+\int_{(b_{n}-\delta,b_{n})}gd\mu_{1}.\]
Let us take $$\epsilon_{0}=\int_{(a_{1},a_{1}+\delta)}gd\mu_{0}+\int_{(b_{n}-\delta,b_{n})}gd\mu_{0},$$ and
$$\epsilon_{1}=\int_{(a_{1},a_{1}+\delta)}gd\mu_{1}+\int_{(b_{n}-\delta,b_{n})}gd\mu_{1},$$ then $$\epsilon_{0}<\frac{2a+k}{a+k}\epsilon, \epsilon_{1}<\frac{2a+k}{a+k}\epsilon.$$ Moreover, one has that
\[\frac{\mu_{0}(X_{n})+\dfrac{a}{a+k}\mu_{0}(X_{1})+\epsilon_{0}}{\mu_{1}(X_{n})+\dfrac{a}{a+k}\mu_{1}(X_{1})+\epsilon_{1}}=\frac{b}{b+k} \] (since $\phi(g)\in C[0,1]_{(b,k)}$).

Then $$|\frac{a}{a+k}\mu_{0}(X_{1})+\mu_{0}(X_{n})-\frac{b}{b+k}(\frac{a}{a+k}\mu_{1}(X_{1})+\mu_{1}(X_{n}))|<r\epsilon$$ for some $r$, which implies that $$\frac{a}{a+k}\mu_{0}(X_{1})+\mu_{0}(X_{n})=\frac{b}{b+k}(\frac{a}{a+k}\mu_{1}(X_{1})+\mu_{1}(X_{n}))$$ since $\epsilon$ is arbitrary.
Similar proofs go through for other cases.
\end{proof}

\begin{cor}\label{lem:mea}
Given any unital positive linear map $\phi$ from $C[0,1]$ to $C[0,1]$ which sends $C[0,1]_{(a,k)}$ to $C[0,1]_{(b,k)}$, denote by $\mu_0$ and $\mu_1$ the measures induced by evaluations of $\phi$ at 0 and 1. Let there be given a partition $\{X_{1},X_{2},...,X_{n}\}$ of $[0,1]$, where $X_{i}$ is a connected Borel set and $0\in X_{1}$, $1\in X_{n}$,
then we have the following distribution of $\mu_0$ and $\mu_1$ with respect to the partition:
 \begin{align}
&\mu_{0}(X_{1})=\frac{b}{b+k}\mu_{1}(X_{1})+\frac{a+k}{b+k},\\
&\mu_{0}(X_{n})=\frac{b}{b+k}\mu_{1}(X_{n})-\frac{a}{b+k}.
\end{align}
\end{cor}
\begin{proof} By adding (3.1) from $i=2,...n-1$ and (3.2), we have
$$\begin{aligned}
&\sum_{i=2}^{n-1}\mu_{0}(X_{i})+\frac{a}{a+k}\mu_{0}(X_{1})+\mu_{0}(X_{n})\\
=&\sum_{i=2}^{n-1}\frac{b}{b+k}\mu_{1}(X_{i})+\frac{b}{b+k}[\frac{a}{a+k}\mu_{1}(X_{1})+\mu_{1}(X_{n})]
\end{aligned}$$
Then we add $\mu_{0}(X_{1})+\dfrac{b}{b+k}\mu_{1}(X_{1})$ from both sides of the equation above, since $\mu_{0}([0,1])=1$, $\mu_{1}([0,1])=1$, we get
$$1+\frac{a}{a+k}\mu_{0}(X_{1})+\frac{b}{b+k}\mu_{1}(X_{1})=\frac{b}{b+k}+\mu_{0}(X_{1})+\frac{b}{b+k}\frac{a}{a+k}\mu_{1}(X_{1}).$$ Then one can solve $\mu_0(X_1)$ from above to get the equation (3.3).
By (3.2) and (3.3), we could get the equation (3.4).
\end{proof}

\begin{remark}\label{b<a} By (3.3), we have $$\mu_{0}(X_{1})=\frac{b}{b+k}\mu_{1}(X_{1})+\frac{a+k}{b+k}\geq\frac{a+k}{b+k},$$ then $1\geq\mu_{0}(X_{1})\geq\dfrac{a+k}{b+k}$, thus we get $b\geq a$. In other words, there is no unital positive linear maps on $C[0,1]$ which sends $C[0,1]_{(a,k)}$ to $C[0,1]_{(b,k)}$ if $b<a$.\\
\end{remark}

\begin{lem}\label{app}
For any $\eta>0$, and for all the given $\mu_{0}(X_i)$, $\mu_{1}(X_i)(i=1,...,n)$ as above in Lemma \ref{mea2}, there exist rational numbers $0\leq r_1,...,r_n\leq1$ and $0\leq s_1,...,s_n\leq1$ which add up to 1 respectively, such that
$$0\leq r_i-\mu_{0}(X_i)\leq\eta,\, 0\leq s_i-\mu_{1}(X_i)\leq\eta,\, i=1, 2,...,n;$$ moreover, the relations among $\mu_{0}(X_i)$ and $\mu_{1}(X_i)$ hold for these $r_i$ and $s_i$, i.e.,
\begin{align}
&r_i=\frac{b}{b+k}s_i, i=2,...,n-1, \\
&\frac{a}{a+k}r_1+r_n=\frac{b}{b+k}(\frac{a}{a+k}s_1+s_n).
\end{align}
\end{lem}
\begin{proof} We take rational approximations $s_i$ for $\mu_1(X_i)$ with $0\leq s_i-\mu_{1}(X_i)\leq\frac{\eta}{n}$, and then take $r_i=\dfrac{b}{b+k}s_i$, then they fit the desired relation because of the relation between $\mu_{0}(X_i)$ and $\mu_1(X_i) (2\leq i\leq n-1)$.
Fix a rational approximation $s_n$ for $\mu_1(X_n)$ with $0\leq s_n-\mu_{1}(X_n)\leq\frac{\eta}{n}$, 
then
$$ \begin{aligned}
 &1-(s_2+...+s_{n-1}-s_n)-\mu_1(X_1)\\
 =&1-(s_2+...+s_{n-1}-s_n)-(1-(\mu_1(X_2)+...+\mu_1(X_{n-1}))-\mu_1(X_n))\\
 =& (s_2-\mu_1(X_2))+...+(s_{n-1}-\mu_1(X_{n-1}))+(s_n-\mu_1(X_n))\leq\eta.
 \end{aligned} $$ We take $$s_1=1-(s_2+...+s_{n-1})-s_n, r_n=\frac{b}{b+k}s_n-\frac{a}{b+k},$$ and $$r_1=1-(r_2+...+r_{n-1})-r_n,$$ then all of the data fit in the requirements.
\end{proof}

\begin{cor}\label{boundary}
For any positive integer $N$, for any collection of $N$ points $\{x_i \in (0,1)|1\leq i\leq N\}$, and for integers $k_1, k_n$ and $m_1, m_n$ satisfying the relation (3.6), Then we have that $$k_1f(0)+\sum\limits_{i=1}^{N}bl_if(x_i)+k_nf(1)=\frac{b}{b+k}(m_1f(0)+\sum\limits_{i=1}^{N}(b+k)l_if(x_i)+m_nf(1))$$ for all $f\in C[0,1]_{(a,k)}$, where these $l_i$ are arbitrarily chosen positive integers.
\end{cor}
\begin{proof} Since $f\in C[0,1]_{(a,k)}$, one has that $f(0)=\dfrac{a}{a+k}f(1)$. Then the left hand side of above equals to $(\dfrac{a}{a+k}k_1+k_n)f(1)+\sum\limits_{i=1}^{N}bl_if(x_i)$. Hence it coincides with the right hand side by the relation (3.6).
\end{proof}


\textbf{Now we proceed to prove Theorem \ref{thm2}:}
\begin{proof}
$\mathit{Step I}$, The first step is exactly the same as the first step of Theorem \ref{thm1}. To avoid redundancy, we skip it and still use the same notations there.

$\mathit{Step II}$, In a similar way, we approximate the finite sum of evaluations by a linear map $w$ on $C[0,1]$ defined as an integral of the composition with some continuous function $h(y,t)$ from $[0,1]\times[0,1]$ to $[0,1]$.


Let us define
$$\begin{aligned}
& l_{1}(y)=\lambda_{2}(y)\\
&l_{2}(y)=\max\{0,\min\{\lambda_{1}(y)-\lambda_{1}(1),\lambda_{2}(1)-\lambda_{2}(y)\}\}\\
&\cdots\\
&l_{2n-3}(y)=\lambda_{n}(y)\\
&l_{2n-2}(y)=\max\{0,\min\{\lambda_{1}(y)-\lambda_{1}(1)-l_{2}(y)-...-l_{2n-4}(y),\lambda_{n}(1)-\lambda_{n}(y)\}\}\\
&l_{2n-1}(y)=1-\sum_{j=1}^{2n-2}l_{j}(y)\\
&l_{2n}(y)=0.\\
\end{aligned}$$
By (3.1-3.4), we have
$$\lambda_{i}(0)=\mu_{0}(X_{i})\leq \mu_{1}(X_{i})= \lambda_{i}(1)(i=2,...n),$$ and
$$\lambda_{1}(0)=\mu_{0}(X_{1})\geq\mu_{1}(X_{1})= \lambda_{1}(1).$$

Thus$$
\begin{aligned}
&l_{1}(0)=\mu_{0}(X_{2}),\qquad \qquad\qquad \;\;\;\; \; \; \; \; \; \; \;\qquad  l_{1}(1)=\mu_{1}(X_{2}),\\
&l_{2}(0)=\mu_{1}(X_{2})-\mu_{0}(X_{2}),\qquad\qquad \;\;\;\; \;\;\;\;\;\;l_{2}(1)=0,\\
&\cdots\\
&l_{2n-3}(0)=\mu_{0}(X_{n}),\qquad\qquad \;\;\qquad\qquad \;\;\;l_{2n-3}(1)=\mu_{1}(X_{n}),\\
&l_{2n-2}(0)=\mu_{1}(X_{n})-\mu_{0}(X_{n}),\qquad\qquad \;\;\;l_{2n-2}(1)=0,\\
&l_{2n-1}(0)=\mu_{1}(X_{1}),\qquad\qquad \qquad\qquad \;\;\;\;\;l_{2n-1}(1)=\mu_{1}(X_{1}),\\
&l_{2n}(0)=0.\qquad\qquad\qquad\qquad \;\;\;\qquad\qquad \;l_{2n}(1)=0.\\
 \end{aligned}$$
Next define $G_{0},G_{1},...G_{2n}:[0,1]\rightarrow[0,1]$ by
\[ G_{0}(y)=0,G_{j}(y)=\sum_{i=1}^{j}\ l_{i}(y),\;j=1,2,...,2n.\]
Then for $j=2k$ we have the consistency that $$G_{j}(0)=G_{j}(1)=\sum_{i=1}^{k}\ \mu_{1}(X_{i}).$$

Let  $\delta>0$  be a rational number, for each $y\in [0,1]$, these points $\{G_i(y)\}_i$ give to a partition of $[0,1]$. 

Moreover, for each $j$, we define
$$f_{j}(y)=\min\{G_{j-1}(y)+\delta;\frac{G_{j-1}(y)+G_{j}(y)}{2}\},$$ and
$$g_{j}(y)=\max\{G_{j}(y)-\delta;\frac{G_{j-1}(y)+G_{j}(y)}{2}\}.$$

To define $h(y,t)$, we only need to define $h(y,t)$ on each $[0,1]\times[G_{j-1},G_{j}(y)]$. Let us denote by $h_{j}(y,t)$ this restriction.
For our purpose, we choose the following $h_j(y,t)$:
\[h_{j}(y,t)=\begin{cases}
                  \frac{x_{j}(t-G_{j-1}(y))}{\delta} & t\in [G_{j-1}(y),f_{j}(y)] \\
                  \min(x_{j},\frac{x_{j}(G_{j}(y)-G_{j-1}(y))}{2\delta}) & t\in [f_{j}(y),g_{j}(y)] \\
                  \frac{x_{j}(G_{j}(y)-t)}{\delta} & t\in [g_{j}(y),G_{j}(y)]. \\
                 \end{cases}
               \]
where $x_{j}=x_{i}$ if $j=2i-1$, $x_{j}=0$ if $j=2i$.

Then $h_{j}(y,t)$ satisfies that for any $y\in [0,1]$,
$$|h_{j}(y,t_{1})-h_{j}(y,t_{2})|\leq\frac{|t_{1}-t_{2}|}{\delta},$$
and $h_{j}(y,t):[0,1]\times[G_{j-1},G_{j}(y)]\rightarrow[0,1]$ is continuous.
Hence $h(y,t):[0,1]\times[0,1]\rightarrow[0,1]$ is continuous and
$$|h(y,t_{1})-h(y,t_{2})|\leq\frac{|t_{1}-t_{2}|}{\delta}.$$

Define $w:C[0,1]\rightarrow C[0,1]$ by $$w(f)(y)=\int_{0}^{1}f(h(y,t))dt$$ where $f\in C[0,1], y\in[0,1]$.

Then for all $f\in F$, one has that
$$\begin{aligned}
&|\sum_{i=1}^n\lambda_{i}(y)f(x_{i})-w(f)(y)|\\
=& |\sum_{j=1}^{2n}l_{j}(y)f(x_{j})-w(f)(y)|\\
=&|\sum_{j=1}^{2n}\int_{G_{j-1}(y)}^{G_{j}(y)}f(x_{j})dt-\int_{0}^{1}f(h(y,t))dt|\\
\leq&\sum_{j=1}^{2n}|\int_{G_{j-1}(y)}^{G_{j}(y)}f(x_{i})-f(h(y,t))dt|\\
=& \sum_{j=1}^{2n}|(\int_{G_{j-1}(y)}^{f_{j}(y)}+\int_{f_{j}(y)}^{g_{j}(y)}+\int_{g_{j}(y)}^{G_{j}(y)}f(x_{j}-f(h(y,t))dt|\\
\leq& \sum_{j=1}^{2n}(2\delta\sup_{f\in F}\|f\|+0+2\delta\sup_{f\in F}\|f\|)\\
=& 8n\delta\sup_{f\in F}\|f\|<\frac{\epsilon}{4}.
\end{aligned}$$
The last inequality holds because of the choice of $\delta$ later.

$\mathit{Step III}$,  We shall choose $N_{1}$ continuous maps on $[0,1]$ to define the homomorphisms. Such maps come from $h(y,t)$ by specifying $N_{1}$ values of $t$.

let $\tau$ be the number of $X_i$ such that $\mu_{0}(X_{i})\neq0$ for $i=2,...,n-1.$ Then $\mu_{1}(X_{i})\neq0$ for the same index by (3.1). We denote these $X_{i}$ by $X_{i_{1}},,...,X_{i_{\tau}}$.

Then  for $X_{1},X_{i_{1}},...,X_{i_{\tau}},X_{n}$, by Lemma \ref{app}, for $\frac{\delta}{n}$,  there exist rational numbers $0\leq r_1,r_{i_{1}}...,r_{i_{\tau}},r_n\leq1$ and $0\leq s_1,s_{i_{1}},...,s_{i_{\tau}},s_n\leq1$ such that
$$\begin{aligned}
&r_1+\sum_{i=1}^{\tau}r_{i_{\iota}}+r_n=1;\, 0\leq r_i-\mu_{0}(X_i)\leq\frac{\delta}{n},i=1,i_{1},...,i_{\tau},n\\
&s_1+\sum_{i=1}^{\tau}s_{i_{\iota}}+s_n=1;\, 0\leq s_i-\mu_{1}(X_i)\leq\frac{\delta}{n},i=1,i_{1},...,i_{\tau},n\\
&\frac{ r_{i_{\iota}}}{ s_{i_{\iota}}}=\frac{b}{b+k}, \iota=1,...\tau\\
&\frac{a}{a+k}r_1+r_n=\frac{b}{b+k}(\frac{a}{a+k}s_1+s_n).\\
\end{aligned}$$
Choose an integer $N_{1}>0$, such that $\frac{1}{N_{1}}<\delta\delta_{0}$ and $N_{1}\delta$, $N_{1}s_{i}$, $N_{1}r_{i}(i=1,i_{1},...i_{\tau},n)$ are integers, let $t_{j}=\frac{j}{N_{1}}\in[0,1]$, $j=1,2,...,N_{1}$.

To save notations, rewrite  $N_{1}s_{i}$, $N_{1}r_{i}$ by $s_{i}$, $r_{i}(i=1,i_{1},...i_{\tau},n)$, then 
 \[\sum_{\iota=1}^{\tau}s_{i_{\iota}}+s_{n}+s_{1}=\sum_{\iota=1}^{\tau}r_{i_{\iota}}+r_{n}+r_{1}=N_{1},\] and
\begin{align}
&0\leq r_i-\mu_{0}(X_i)N_{1}\leq N_{1}\frac{\delta}{n},(i=1,i_{1},...i_{\tau},n)\\
&0\leq s_i-\mu_{1}(X_i)N_{1}\leq N_{1}\frac{\delta}{n}, (i=1,i_{1},...i_{\tau},n)\\
&\frac{ r_{i_{\iota}}}{ s_{i_{\iota}}}=\frac{b}{b+k},\;\iota=1,...\tau\\
&\frac{a}{a+k} r_{1}+r_{n}=\frac{b}{b+k}(\frac{a}{a+k} s_{1}+s_{n}).
\end{align}

Define $\phi_{j}:C[0,1]\rightarrow C[0,1]$ by $\phi_{j}(f)(y)=f(h(y,t_{j}))$, Then
\[w(f)(y)=\int_{0}^{1}f(h(y,t))dt=\sum_{j=1}^{N_{1}}\int_{t_{j-1}}^{t_{j}}f(h(y,t))dt\]
and $$\begin{aligned}
&|w(f)(y)-\frac{1}{N_{1}}\sum_{j=1}^{N_{1}}\phi_{j}(f)(y)|\\ 
=&|\sum_{j=1}^{N_{1}}\int_{t_{j-1}}^{t_{j}}f(h(y,t))-f(h(y,t_{j}))dt|\\
<&\sum_{j=1}^{N_{1}}\int_{t_{j-1}}^{t_{j}}\frac{\epsilon}{4}=\frac{\epsilon}{4}.\\
\end{aligned}$$
for all $f\in F$, where $|f(h(y,t))-f(h(y,t_{j}))|<\frac{\epsilon}{4}$, because $|h(y,t)-h(y,t_{j})|<\delta_{0}$ (Note that $|t-t_{j}|\leq\dfrac{1}{N_{1}}<\delta\delta_{0}$).

$\mathit{Step IV}$, We need to choose some $t_j=\dfrac{j}{N_1}$ among all of them to obtain new $N$ points, such that the corresponding average $\dfrac{1}{N}\sum\limits_i\phi_{i}$ sends $C[0,1]_{(a,k)}$ to $C[0,1]_{(b,k)}$, and also this new average approximates the average of the original $N_1$ homomorphisms. But in the current situation, rather than that for Theorem \ref{thm1}, it is much more complicated.

Define new integers

$\qquad S_{i_{0}}=0,\;\;\;S_{i_{\gamma}}=\sum_{\iota=1}^{\gamma}s_{i_{\iota}}(\gamma=1,...,\tau),\;S_{n}=S_{i_{\tau}}+s_{n}.$

$ \qquad R_{i_{1}}=r_{i_{1}},\;R_{i_{\gamma}}=S_{i_{\gamma-1}}+r_{i_{\gamma}}(\gamma=1,...,\tau),\;R_{n}=S_{i_{\tau}}+r_{n}.$\\
Then
$$\begin{aligned}
&0\leq |S_{i_{\gamma}}-\sum_{\iota=1}^{\gamma}\mu_{1}(X_{i_{\iota}})N_{1}|\leq N_{1}\delta,\gamma=1,...\tau,\\
&0\leq |S_{n}-[\sum_{\iota=1}^{\tau}\mu_{1}(X_{i_{\iota}})+\mu_{1}(X_{n})]N_{1}|\leq N_{1}\delta,\\
&0\leq |R_{i_{\gamma}}-[\sum_{\iota=1}^{\gamma-1}\mu_{1}(X_{i_{\iota}})+\mu_{0}(X_{i_{\gamma}})]N_{1}|\leq N_{1}\delta,\gamma=1,...\tau,\\
&0\leq| R_{n}-[\sum_{\iota=1}^{\tau}\mu_{1}(X_{i_{\iota}})+\mu_{0}(X_{n})]N_{1}|\leq N_{1}\delta.\\
\end{aligned}$$
To ensure $\dfrac{1}{N}\sum\limits_i\phi_{i}$ sends $C[0,1]_{(a,k)}$ to $C[0,1]_{(b,k)} $, we need to exclude some functions $h(y,t_{j})$, for which $h(0,t_{j})\neq x_{i_{\iota}}$ when $j\in(S_{i_{\iota-1}},R_{i_{\iota}}]$ and $h(1,t_{j})\neq x_{i_{\iota}}$ when $j\in(S_{i_{\iota-1}},S_{i_{\iota}}](\iota=1,...,\tau)$; $h(0,t_{j})\neq 1$ when $j\in(S_{i_{\tau}},R_{n}]$ and $h(1,t_{j})\neq1$ when $j\in(S_{i_{\tau}},S_{n}]$.

Assume that we throw out $m_{i_{\iota}}$ functions $h(0,t_j)$ for $j\in(S_{i_{\iota-1}},R_{i_{\iota}}](\iota=1,...,\tau)$, $m_{n}$ functions $h(0,t_j)$ for $j\in(S_{i_{\tau}},R_{n}]$, and $m_{1}$ functions $h(0,t_j)$ for the remaining $j$; $z_{i_{\iota}}$ functions $h(1,t_j)$ for $j\in(S_{i_{\iota-1}},S_{i_{\iota}}](\iota=1,...,\tau)$, $z_{n}$ functions $h(1,t_j)$ for $j\in(S_{i_{\tau}},S_{n}]$, and $z_{1}$ functions $h(1,t_j)$ for $j\in(S_{n},N_{1}]$.

By Corollary \ref{boundary}, to fit in our purpose, we need to throw out functions in proportion, such that the remaining ones could satisfy the relation (3.9-10). So we need to require
 \begin{align}
&\frac{m_{i_{\iota}}}{z_{i_{\iota}}}=\frac{b}{b+k}, \iota=1,...,\tau\\
&\frac{a}{a+k}m_{1}+m_{n}=\frac{b}{b+k}(\frac{a}{a+k}z_{1}+z_{n}).
\end{align}

Next we proceed in two cases.

Case I: $\mu_{0}(X_{n})=0$.

By (3.7), we get $ r_{n}\leq N_{1}\frac{\delta}{n}$. then take $m_{n}=0$. So we suppose
$$\begin{aligned}
&m_{i_{\iota}}=2\delta N_{1}b,\qquad\qquad \qquad\qquad z_{i_{\iota}}=2\delta N_{1}(b+k),\iota=1,...,\tau,\\
&m_{n}=0,\qquad\qquad \qquad\qquad \qquad\; z_{n}=t,\\
&m_{1}=2\delta N_{1}k\tau+t,\qquad  \qquad\;\;\;\;\;\;\;z_{1}=0.\\
\end{aligned}$$
By (3.12), since
\[\frac{a}{a+k}(2\delta N_{1}k\tau+t)=\frac{b}{b+k}t,\]
one has $t=\dfrac{2\delta N_{1}a\tau(b+k)}{b-a}.$
Then we take
$$\begin{aligned}
&m_{i_{\iota}}=2\delta N_{1}(b-a)b,\qquad\qquad  z_{i_{\iota}}=2\delta N_{1}(b-a)(b+k),\iota=1,...,\tau,\\
&m_{n}=0,\qquad\qquad \qquad\qquad \;\;\;\; z_{n}=2\delta N_{1}a\tau(b+k),\\
&m_{1}=2\delta N_{1}(a+k)\tau b,\qquad\;\;\;\;\;   z_{1}=0.\\
\end{aligned}$$
Let $\delta$ satisfy
$$\begin{aligned}
&4n\delta(b+k)b\sup_{f\in F}\|f\|\leq\frac{\epsilon}{4},\\
&3\delta b(b-a)\leq\mu_{0}(X_{i_{\iota}}),\iota=1,2,...\tau,\\
&3\delta (a+k)b\tau\leq\mu_{0}(X_{1}),\\
&3\delta(b+k)a\tau\leq\mu_{1}(X_{n}).\\
\end{aligned}$$
By (3.7-3.8), we have
$$\begin{aligned}
&r_{i_{\iota}}> 2\delta N_{1}(b-a)b,\iota=1,2,...,\tau,\\
&r_{1}> 2\delta N_{1}(a+k)b\tau,\\
&s_{i_{\iota}}>2\delta N_{1}(b-a)(b+k),\iota=1,2,...,\tau,\\
&s_{n}> 2\delta N_{1}(b+k)a\tau.\\
\end{aligned}$$
Consider the set $\Lambda_1$ of integers $j$ which belongs to one of the following intervals:
$$\begin{aligned}
&S_{i_{\gamma-1}}+\delta N_{1}(b-a)b+1\leq j\leq R_{i_{\gamma}}-\delta N_{1}(b-a)b,\gamma=1,...,\tau,\\
&R_{i_{\gamma}}+1\leq j\leq S_{i_{\gamma}}-\delta N_{1}(b-a)k,\gamma=1,...,\tau,\\
&S_{i_{\tau}}+\delta N_{1}(b+k)a\tau+1\leq j\leq S_{n}-\delta N_{1}(b+k)a\tau,\\
&S_{n}+1\leq j\leq N_{1}.\\
\end{aligned}$$
Set $D_1=\{\frac{j}{N_1}\mid j\in\Lambda_1\}$ and $N=|D_1|$, then by the construction of $h(y,t)$, point evaluations of $h$ at the points in $D_1$ has the following conclusion:
$$\begin{aligned}
&h(0,t_{j})\!=\!h(1,t_{j})=\!x_{i_{\gamma}}, \,\text{if}\, S_{i_{\gamma-1}}\!+\!\delta N_{1}(b-a)b\!+\!1\leq \!j\!\leq R_{i_{\gamma}}\!-\!\delta N_{1}(b-a)b\!,\gamma=1,...,\tau\\
&h(0,t_{j})=0,h(1,t_{j})=x_{i_{\gamma}}, \,\text{if}\, R_{i_{\gamma}}+1\leq j\leq S_{i_{\gamma}}-\delta N_{1}(b-a)k,\gamma=1,...,\tau\\
&h(0,t_{j})=0,h(1,t_{j})=1, \,\text{if} \,S_{i_{\tau}}+\delta N_{1}(b+k)a\tau+1\leq j\leq S_{n}-\delta N_{1}(b+k)a\tau,\\
&h(0,t_{j})=0=h(1,t_{j}), \,\text{if}\,S_{n}+1\leq j\leq N_{1}.\\
\end{aligned}$$
 Note that $N=N_{1}(1-2\delta\tau(b+k)b)$, and one has that
 $$\frac{1}{N}\sum_{d=1}^N\phi_{d}(f)(0)=\frac{1}{N}\sum_{d=1}^{N}f(h(0,t_{d}))=\frac{1}{N}[\sum_{\iota=1}^{\tau}(r_{i_{\iota}}-m_{i_{\iota}})f(x_{i_{\iota}})+
(r_{1}-m_{1})f(0)],$$ and
$$\begin{aligned}
&\frac{1}{N}\sum_{d=1}^N\phi_{d}(f)(1)=\frac{1}{N}\sum_{d=1}^{N}f(h(1,t_{d}))\\
=&\frac{1}{N}[\sum_{\iota=1}^{\tau}(s_{i_{\iota}}-z_{i_{\iota}})f(x_{i_{\iota}})+(s_{n}-z_{n})f(1)+s_{1}f(0)].\\
\end{aligned}$$
Thus  $$\frac{1}{N}\sum_{d=1}^N\phi_{d}(f)(0)=\frac{b}{b+k}\frac{1}{N}\sum_{d=1}^N\phi_{d}(f)(1)$$ because of (3.9-3.12) and Corollary \ref{boundary}.

Next, we show that the average of these $N$ homomorphisms will approximate the average of the original $N_{1}$ homomorphisms:
$$\begin{aligned}
&|\frac{1}{N_{1}}\sum_{j=1}^{N_{1}}f(h(y,t_{j}))-\frac{1}{N}\sum_{d=1}^{N}f(h(y,t_{d}))|\\
=&|\frac{1}{N_{1}}\sum_{j=1}^{N_{1}}f(h(y,t_{j}))-\frac{1}{N_{1}(1-2\delta\tau(b+k)b)}\sum_{d=1}^{N}f(h(y,t_{d}))|\\
=&\frac{1}{N_{1}}|\sum_{j=1}^{N_{1}}f(h(y,t_{j}))-\frac{1}{(1-2\delta\tau(b+k)b)}\sum_{d=1}^{N}f(h(y,t_{d}))|\\
=&\frac{1}{N_{1}}|\sum_{d\notin D_{1}}f(h(y,t_{d})+(1-\frac{1}{1-2\delta\tau(b+k)b})\sum_{d=1}^{N}f(h(y,t_{d}))|\\
\leq&\frac{1}{N_{1}}(N_{1}2\delta\tau(b+k)b \sup_{f\in F}\|f\|+N_{1}2\delta\tau(b+k)b\sup_{f\in F}\|f\|)\\
=&4\delta\tau(b+k)b\sup_{f\in F}\|f\|\leq\frac{\epsilon}{4}\;(since\;\tau\leq n-2).\\
\end{aligned}$$

Case II: $\mu_{0}(X_{n})\neq0$.

Let $\delta$ satisfy 
 $$\begin{aligned}
&4\delta n(k+b)b\sup_{f\in F}\|f\|\leq\frac{\epsilon}{4},\\
&3\delta (b-a)b<\mu_{0}(X_{i_{\iota}}),\iota=1,2,...\tau,\\ 
&3\delta (b-a)b<\mu_{0}(X_{n}),\, 3\delta (a+k)(b+b\tau)<\mu_{0}(X_{1}),\\
&3\delta(b+k)(a\tau+b)<\mu_{1}(X_{n}).\\
\end{aligned}$$
By (3.7-3.8), we have
$$\begin{aligned}
&r_{i_{\iota}}> 2\delta N_{1}(b-a)b,\iota=1,2,...,\tau,\\
&r_{n}> 2\delta N_{1}(b-a)b,\\
&r_{1}>2\delta N_{1}(a+k)(1+\tau)b,\\
&s_{i_{\iota}}> 2\delta N_{1}(b-a)(b+k),\iota=1,2,...,\tau,\\
&s_{n}> 2\delta N_{1}(b+k)(a\tau+b).\\
\end{aligned}$$
So we suppose
$$\begin{aligned}
&m_{i_{\iota}}=2\delta N_{1}(b-a)b,\qquad\qquad z_{i_{\iota}}=2\delta N_{1}(b-a)(b+k),\iota=1,...,\tau,\\
&m_{n}=2\delta N_{1}(b-a)b,  \qquad\qquad z_{n}=2\delta N_{1}(b-a)b+t,\\
&m_{1}=2\delta N_{1}k\tau(b-a)+t,  \qquad z_{1}=0.\\
\end{aligned}$$
By (3.8), since
\[\frac{a}{a+k}[2\delta N_{1}k\tau(b-a)+t]+2\delta N_{1}(b-a)b=\frac{b}{b+k}[2\delta N_{1}(b-a)b+t],\]\\
one has $t=2\delta N_{1}[a\tau(b+k)+b(a+k)]$.
Then we take
$$\begin{aligned}
&m_{i_{\iota}}=2\delta N_{1}(b-a)b,\qquad\qquad z_{i_{\iota}}=2\delta N_{1}(b-a)(b+k),\iota=1,...,\tau,\\
&m_{n}=2\delta N_{1}(b-a)b,\qquad\qquad z_{n}=2\delta N_{1}(a\tau+b)(b+k),\\
&m_{1}=2\delta N_{1}(a+k)(a+\tau)b,\;\;\;z_{1}=0.\\
\end{aligned}$$
Consider the set $\Lambda_2$ of integers $j$ which belongs to one of the following intervals:
$$\begin{aligned}
&S_{i_{\gamma-1}}+\delta N_{1}(b-a)b+1\leq j\leq R_{i_{\gamma}}-2\delta N_{1}(b-a)b,\gamma=1,...,\tau,\\
&R_{i_{\gamma}}+1\leq j\leq S_{i_{\gamma}}-\delta N_{1}(b-a)k,\gamma=1,...,\tau,\\
&S_{i_{\tau}}+\delta N_{1}(b-a)b+1\leq j\leq R_{n}-\delta N_{1}(b-a)b,\\
&R_{n}+1\leq j\leq S_{n}-2\delta N_{2}(ab\tau+ak\tau+bk+ba),\\
&S_{n}+1\leq j\leq N_{1}.\\
\end{aligned}$$
Set $D_2=\{\frac{j}{N_1}\mid j\in\Lambda_2\}$ and $N=|D_2|$, then by the construction of $h(y,t)$, point evaluations of $h$ at the points in $D_2$ has the following conclusion:
$$\begin{aligned}
 &h(0,t_{j})=h(1,t_{j})=x_{i_{\gamma}}\!, \,\text{if}\,S_{i_{\gamma-1}}\!+\!\delta N_{1}(b-a)b\!+\!1\leq \!j\!\leq R_{i_{\gamma}}\!-\!\delta N_{1}(b-a)b,\gamma=1,...,\tau\\
 &h(0,t_{j})=0,h(1,t_{j})=x_{i_{\gamma}}, \,\text{if}\,R_{i_{\gamma}}+1\leq j\leq S_{i_{\gamma}}-\delta N_{1}(b-a)k,\gamma=1,...,\tau\\
 &h(0,t_{j})=h(1,t_{j})=1, \,\text{if}\,S_{i_{\tau}}+\delta N_{1}(b-a)b+1\leq j\leq R_{n}-\delta N_{1}(b-a)b,\\
 &h(0,t_{j})=0,h(1,t_{j})=1, \,\text{if}\,R_{i_{\tau}}++1\leq j\leq S_{n}-2\delta N_{1}(ab\tau+ak\tau+bk+ba),\\
 &h(0,t_{j})=0=h(1,t_{j}), \,\text{if}\,S_{n}+1\leq j\leq N_{1}.\\
\end{aligned}$$
Note that $N=N_{1}(1-4\delta(1+\tau)(k+b)b)$, one has that 
 $$\frac{1}{N}\sum_{d=1}^N\phi_{d}(f)(0)=\frac{1}{N}[\sum_{\iota=1}^{\tau}(r_{i_{\iota}}-m_{i_{\iota}})f(x_{i_{\iota}})+(r_{n}-m_{n})f(1)+(r_{1}-m_{1})f(0)],$$
 and
$$\frac{1}{N}\sum_{d=1}^N\phi_{d}(f)(1)=\frac{1}{N}[\sum_{\iota=1}^{\tau}(s_{i_{\iota}}-z_{i_{\iota}})f(x_{i_{\iota}})+(s_{n}-z_{n})f(1)+s_{1}f(0)].$$
Thus  $$\frac{1}{N}\sum_{d=1}^N\phi_{d}(f)(0)=\frac{b}{b+k}\frac{1}{N}\sum_{d=1}^N\phi_{d}(f)(1)$$ because of (3.9-3.12) and Corollary \ref{boundary}.

Next, we show the average of these $N$ homomorphisms will approximate the average of the original $N_{1}$ homomorphisms:
$$\begin{aligned}
&|\frac{1}{N_{1}}\sum_{j=1}^{N_{1}}f(h(y,t_{j}))-\frac{1}{N}\sum_{d=1}^{N}f(h(y,t_{d}))|\\
=&|\frac{1}{N_{1}}\sum_{j=1}^{N_{1}}f(h(y,t_{k}))-\frac{1}{N_{1}(1-2\delta(1+\tau)(k+b)b}\sum_{d=1}^{N}f(h(y,t_{d}))|\\
=&\frac{1}{N_{1}}|\sum_{j=1}^{N_{1}}f(h(y,t_{j}))-\frac{1}{1-2\delta(1+\tau)(k+b)b}\sum_{d=1}^{N}f(h(y,t_{d}))|\\
=&\frac{1}{N_{1}}|\sum_{d\notin D_{2}}f(h(y,t_{d})+(1-\frac{1}{1-2\delta(1+\tau)(k+b)b}\sum_{d=1}^{N}f(h(y,t_{d}))|\\
\leq&\frac{1}{N_{2}}(N_{1}2\delta(1+\tau)(k+b)b\sup_{f\in F}\|f\|+N_{1}2\delta(1+\tau)(k+b)b)\sup_{f\in F}\|f\|)\\
=&4\delta(1+\tau)(k+b)b \sup_{f\in F}\|f\|\leq\frac{\epsilon}{4} \;(since\; \tau\leq n-2)\\
\end{aligned}$$

In other words, no matter in which cases we can always find $N$ functions $h(y, t_d), d=1,...,N$, as required.

Finally, let us define $\phi_{d}:C[0,1]\rightarrow C[0,1]$ by $\phi_{d}(f(y))=f(h(y,t_{d}))$, for $d=1,...,N$, then
$$\begin{aligned}
&|\phi(f)(y)-\frac{1}{N}\sum_{d=1}^{N}\phi_{j}(f)(y)|\\
=&|\phi(f)(y)-\frac{1}{N}\sum_{d=1}^{N}f(h(y,t_{d}))|\\
\leq&|\phi(f)(y)-\sum_{i=1}^{n}\lambda_{i}(y)f(x_{i})|+|\sum_{i=1}^n\lambda_{i}(y)f(x_{i})-w(f)(y)|\\
+&|w(f)(y)-\frac{1}{N_{1}}\sum_{j=1}^{N_{1}}f(h(y,t_{j}))|+|\frac{1}{N_{1}}\sum_{j=1}^{N_{1}}f(h(y,t_{j}))-\frac{1}{N}\sum_{d=1}^{N}f( h(y,t_{d}))|\\
\leq&\frac{\epsilon}{4}+\frac{\epsilon}{4}+\frac{\epsilon}{4}+\frac{\epsilon}{4}\leq\epsilon\\
\end{aligned}$$
\end{proof}

\begin{remark} 
1, In the proof of Theorem \ref{thm2},  rather than Theorem \ref{thm1}, we could only achieve that the number of homomorphisms depends on the finite subset $F$, $\epsilon$ and $\phi$.

2, In an overview of the approximation procedure above, one can see that the crucial thing is the distribution of the measures induced by evaluations of $\phi$ at 0 and 1, which determines the boundary condition at algebra level. This idea will go through in general cases.

\end{remark}

Next, we consider general subspaces of $C[0,1]$: $$C[0,1]_\alpha=\{f\in C[0,1] \mid f(0)=\alpha f(1)\},$$ where $0< \alpha<1$. Obviously, it is also a subspace of co-dimension one. Applying similar proofs as in Lemma \ref{point mass}, Lemma \ref{mea2},  Corollary \ref{lem:mea} and Corollary \ref{boundary}, parallel results hold in this general setting, namely, the corresponding measures induced by evaluations of a Markov operator at 0 and 1 have similar features.

\begin{lem}\label{mea3}
Given any unital positive linear map $\phi$ on $C[0,1]$ which sends $C[0,1]_{\alpha}$ to $C[0,1]_{\beta}$, denote by $\mu_{0}$ and $\mu_{1}$ the measures induced by evaluations of $\phi$ at 0 and 1. Let there be given a partition $\{X_{1},X_{2},...,X_{n}\}$ of $[0,1]$, where $X_{i}$ is a connected Borel set and $0\in X_{1},1\in X_{n}$.
Then we have the following distribution of $\mu_{0}$ and $\mu_{1}$ with respect to the partition:
$$\begin{aligned}
&\mu_{0}(X_{i})=\beta\mu_{1}(X_{i}) \;for \;i=2,...,n-1,\\
&\alpha\mu_{0}(X_{1})+\mu_{0}(X_{n})=\beta[\alpha\mu_{1}(X_{1})+\mu_{1}(X_{n})].\\
\end{aligned}$$
\end{lem}

\begin{cor}\label{lem:gmea1}
Given any unital positive linear map $\phi$ from $C[0,1]$ to $C[0,1]$  which sends $C[0,1]_{\alpha}$ to $C[0,1]_{\beta}$, denote by $\mu_{0}$ and $\mu_{1}$ the measures induced by evaluations of $\phi$ at 0 and 1. Let there be given a partition  $\{X_{1},X_{2},...,X_{n}\}$ of $[0,1]$, where $X_{i}$ is a connected Borel set and $0\in X_{1},1\in X_{n}$.
Then we have the following distribution of $\mu_0$ and $\mu_1$ with respect to the partition:
$$\begin{aligned}
&\mu_{0}(X_{1})=\beta\mu_{1}(X_{1})+\frac{1-\beta}{1-\alpha},\\
&\mu_{0}(X_{n})=\beta\mu_{1}(X_{n})-\frac{1-\beta}{1-\alpha}\alpha.\\
\end{aligned}$$
In particularly, if $\alpha=\beta$, then $\phi(f)(0)=f(0), \phi(f)(1)=f(1)$ for all $f\in C[0,1]$.
\end{cor}

Once the distribution of the induced measures at 0 and 1 has similar features, then the corresponding approximation results hold.  

\begin{thrm}\label{same} Given any finite subset $ F \subset C[0,1]_\alpha$ and $\epsilon>0$, there is an integer $ N>0$ with the following property:
for any unital positive linear map $\phi$ on $C[0,1]$ which preserves $C[0,1]_\alpha$, there are N homomorphisms $\phi_{1},\phi_{2},...,\phi_{N}$ from $C[0,1]$ to $C[0,1]$ such that $\dfrac{1}{N}\sum\limits_{i=1}^{N}\phi_{i}(f)\in C[0,1]_\alpha$ for all $f\in C[0,1]_\alpha$ and
$$\|\phi(f)-\frac{1}{N}\sum_{i=1}^N\phi_{i}(f)\| < \epsilon$$ for all $f\in F$.
\end{thrm}
\begin{proof}
Because the measures induced by evaluations of $\phi$ at 0 and 1 are still point mass measures by Corollary \ref{lem:gmea1}, similar proofs as in Theorem \ref{thm1} go through. 
\end{proof}
 
\begin{thrm}\label{rational} Let $\alpha$, $\beta$ be rational numbers and $\beta>\alpha$, given any finite subset $ F \subset C[0,1]_\alpha$ and $\epsilon>0$, for any unital positive linear map $\phi$ on $C[0,1]$ which sends $C[0,1]_\alpha$ to $C[0,1]_\beta$ , there are N homomorphisms $\phi_{1},\phi_{2},...,\phi_{N}$ from $C[0,1]$ to $C[0,1]$ such that $\dfrac{1}{N}\sum_{i=1}^N\phi_{i}(f)\in C[0,1]_\beta$ for all $f\in C[0,1]_\alpha$ and
$$\|\phi(f)-\frac{1}{N}\sum_{i=1}^N\phi_{i}(f)\| < \epsilon$$ for all $f\in F$.
\end{thrm}
\begin{proof}

Based on Lemma \ref{mea3} and Corollary \ref{lem:gmea1}, parallel results as Lemma \ref{mea2},  Corollary \ref{lem:mea} and Corollary \ref{boundary} hold, then similar proofs as in Theorem \ref{thm2} go through.
\end{proof}

\begin{remark}
Because of some technicalities in the approximation procedure, mainly about some integer issue at certain points, we restrict our attentions on the case that $\alpha, \beta$ being rational numbers in the theorem above.

\end{remark}

\section{generalizations}

In general, we replace $[0,1]$ by compact metrizable spaces $X$ and $Y$. Let $p,q$ and $z, w$ be fixed points in $X$ and $Y$, $\alpha, \beta\in(0,1)$, and consider subspaces:$$C(X)_{(p, q, \alpha)}=\{f\in C(X) \mid f(p)=\alpha f(q)\},$$ and $$C(Y)_{(z, w, \beta)}=\{f\in C(Y) \mid f(z)=\beta f(w)\}.$$

Similarly, in this case, the corresponding measures have  the same features.
\begin{lem}\label{mea4}
Let there be given compact metrizable spaces $X$ and $Y$. Given any unital positive linear map $\phi$ from $C(X)$ to $C(Y)$ which sends $C(X)_{(p,q,\alpha)}$ to $C(Y)_{(z,w,\beta)}$, denote by $\mu_{z}$, $\mu_{w}$ the measures induced by evaluations of $\phi$ at $z$ and $w$. Let there be given a partition $\{X_{1},X_{2},...,X_{n}\}$ of $X$, where $X_{i}$ is a connected Borel set and $p\in X_{1},q\in X_{n}$.

Then we have the following distribution of $\mu_{z}$ and $\mu_{w}$ with respect to the partition:
$$\begin{aligned}
&\mu_{z}(X_{i})=\beta\mu_{w}(X_{i}) \;for \;i=2,...,n-1,\\
&\alpha\mu_{z}(X_{1})+\mu_{z}(X_{n})=\beta[\alpha\mu_{w}(X_{1})+\mu_{w}(X_{n})].\\
\end{aligned}$$
\end{lem}
\begin{proof} First of all we prove the first relation for $i=2,...,n-1.$

 By the regularity of $\mu_{z},\mu_{w}$, for any $\epsilon>0$, there exist open sets $O_{i0},O_{i1}\supseteq X_{i}$, closed sets $F_{i0},F_{i1}\subseteq X_{i}$ with $ \mu_{z}(O_{i0}\setminus F_{i0})<\epsilon $, $ \mu_{w}(O_{i1}\setminus F_{i1})<\epsilon $.
Let $F_{i}=F_{i0}\cup F_{i1}\subseteq X_{i}$, then $F_{i}$ is closed and $\mu_{z}(X_{i}\setminus F_{i})<\epsilon $, $ \mu_{w}(X_{i}\setminus F_{i})<\epsilon$.
Since $p\notin X_{i}, q\notin X_{i}$, then $p\notin F_{i}, q\notin F_{i}$, then there are open sets $Q_{i0},Q_{i1}\supseteq F_{i}$, such that $p\notin Q_{i0}$, $q\notin Q_{i1}$.
Let $O_{i}=O_{i0}\cap O_{i1}\cap Q_{i0}\cap Q_{i1} $, then $O_{i}$ is open ,$O_{i}\supseteq F_{i}$ and $\mu_{z}(O_{i}\setminus F_{i})<\epsilon $, $ \mu_{w}(O_{i}\setminus F_{i})<\epsilon $.
Since $X$ is a normal space, and $ O_{i}^{c}\cap F_{i}=\varnothing$, by Urysohn's Lemma, there is a continuous function $f_{i}:X\rightarrow[0,1]$ such that $f_{i}(x)={1}$ if $x\in F_{i}$, $f_{i}(x)={0}$ if $x\in O_{i}^{c}$; moreover, $f_{i}\in C(X)_{(p,q,\alpha)}$ since $p, q\notin O_{i}$.
Then $\phi(f_{i})\in C(Y)_{(z,w,\beta)}$, while
  \[\phi(f_{i})(z)=\int_{X}f_{i}d\mu_{z}=\mu_{z}(F_{i})+\int_{(O_{i}\setminus F_{i})}f_{i}d\mu_{z},\]
  \[\phi(f_{i})(w)=\int_{X}f_{i}d\mu_{w}=\mu_{w}(F_{i})+\int_{O_{i}\setminus F_{i})}f_{i}d\mu_{w},\] so
\[\beta=\frac{\phi(f_{i})(z)}{\phi(f_{i})(w)}=\frac{\mu_{z}(F_{i})+\int_{(O{i}\setminus F_{i})}f_{i}d\mu_{z}}{\mu_{w}(F_{i})+\int_{(O_{i}\setminus F_{i})}f_{i}d\mu_{w}}.\]
 Then $$\beta\mu_{w}(F_{i})+\beta\int_{(O_{i}\setminus F_{i})}f_{i}d\mu_{w}=\mu_{z}(F_{i})+\int_{(O_{i}\setminus F_{i})}f_{i}d\mu_{z},$$ and
$$\begin{aligned}
&|\beta\mu_{w}(X_{i})-\mu_{z}(X_{i})|\\
=&|\beta[\mu_{w}(X_{i})-\mu_{w}(F_{i})]-[\mu_{z}(X_{i})-\mu_{z}(F_{i})]\\
+&\int_{(O_{i}\setminus F_{i})}f_{i}d\mu_{z}-\beta\int_{(O_{i}\setminus F_{i})}f_{i}d\mu_{w}|\\
=&|\beta\mu_{w}(X_{i}\setminus F_{i})-(\mu_{z}(X_{i}\setminus F_{i})+\int_{(O_{i}\setminus F_{i})}f_{i}d\mu_{z}-\beta\int_{(O_{i}\setminus F_{i})}f_{i}d\mu_{w}|\\
\leq&4\varepsilon.\\
 \end{aligned}$$
Then $ \beta\mu_{w}(X_{i})=\mu_{z}(X_{i})$  when $\varepsilon$ tends to 0, $ i=2,...,n-1$.

Secondly, we prove the second relation for $i=1$ and $i=n.$

There exist open sets $O_{1},O_{n}$, closed sets $F_{1},F_{n}$, with $O_{1}\supseteq X_{1}\supseteq F_{1},O_{n}\supseteq X_{n}\supseteq F_{n}$, such that $ \mu_{z}(O_{1}\setminus F_{1})<\epsilon, \mu_{z}(O_{n}\setminus F_{n})<\epsilon, \mu_{w}(O_{1}\setminus F_{1})<\epsilon, \mu_{w}(O_{n}\setminus F_{n})<\epsilon.$ We have $F_{1}\cap F_{n}=\varnothing$ since $X_{1}\cap X_{n}=\varnothing$, then there exist open sets $ Q_{1}\supseteq F_{1}, Q_{n}\supseteq F_{n}$, such that $Q_{1}\cap Q_{n}=\varnothing$. Let us set $A_{1}=O_{1}\cap Q_{1},A_{n}=O_{n}\cap Q_{n}$, then $A_{1}\supseteq F_{1},  A_{n}\supseteq F_{n}$. By Urysohn's Lemma, there is a continuous function $f_{1}:X\rightarrow[0,1]$ such that $f_{1}(x)={\alpha}$ if $x\in F_{1}$, $f_{1}(x)={0}$ if $x\in A_{1}^{c}$; and a continuous function $f_{n}:X\rightarrow[0,1]$ such that $f_{n}(x)={1}$ if $x\in F_{n}$, $f_{n}(x)={0}$ if $x\in A_{n}^{c}$.

Let $f=f_{1}+f_{n}$, then $f$ is a continuous function from $X$ to $ [0,1]$.
We have $f(x)=f_{1}(x)+f_{n}(x)=\alpha+0=\alpha$ if $x\in F_{1}$, $f(x)=f_{1}(x)+f_{n}(x)=0+1=1$ if $x\in F_{n}$, $f(x)=f_{1}(x)+f_{n}(x)=0$ if $x\in A_{1}^{c}\cap A_{n}^{c}$.

Suppose $p\in F_{1}, q\in F_{n}$, if not, we add $p$ into $ F_{1}$, $q$ into $F_{n}$, then $f\in C(X)_{(p,q,\alpha)}$, and $\phi(f)\in C(Y)_{(z,w,\beta)}$, while
\[\phi(f)(z)=\int_{X}fd\mu_{z}=\alpha\mu_{z}(F_{1})+\mu_{z}(F_{n})+\int_{(A_{1}\setminus F_{1})}fd\mu_{z}+\int_{(A_{n}\setminus F_{n})}fd\mu_{z},\]

\[\phi(f)(w)=\int_{X}fd\mu_{w}=\alpha\mu_{w}(F_{1})+\mu_{w}(F_{n})+\int_{(A_{1}\setminus F_{1})}fd\mu_{w}+\int_{(A_{n}\setminus F_{n})}fd\mu_{w},\]
so \[\beta=\frac{\phi(f)(z)}{\phi(f)(w)}=\frac{\alpha\mu_{z}(F_{1})+\mu_{z}(F_{n})+\int_{(A_{1}\setminus F_{1})}fd\mu_{z}+\int_{(A_{n}\setminus F_{n})}fd\mu_{z}}{\alpha\mu_{w}(F_{1})+\mu_{w}(F_{n})+\int_{(A_{1}\setminus F_{1})}fd\mu_{w}+\int_{(A_{n}\setminus F_{n})}fd\mu_{w}}.\]
Then
$$\begin{aligned}
 &|\alpha\mu_{z}(X_{1})+\mu_{z}(X_{n})-\beta[\alpha\mu_{w}(X_{1})+\mu_{w}(X_{n})]|\\
 =&|\alpha[\mu_{z}(X_{1})-\mu_{z}(F_{1})]+[\mu_{z}(X_{n})-\mu_{z}(F_{n})]\\
 -&[\int_{(A_{1}\setminus F_{1})}fd\mu_{z}+\int_{(A_{n}\setminus F_{n})}fd\mu_{z}]\\
 -&\beta\{\alpha[\mu_{w}(X_{1})-\mu_{w}(F_{1})]+[\mu_{w}(X_{n})-\mu_{w}(F_{n})]\\
 -&[\int_{(A_{1}\setminus F_{1})}fd\mu_{w}+\int_{(A_{n}\setminus F_{n})}fd\mu_{w}]\}|\\
 \leq&8\epsilon\\
\end{aligned}$$
Thus $\alpha\mu_{z}(X_{1})+\mu_{z}(X_{n})=\beta(\alpha\mu_{w}(X_{1})+\mu_{w}(X_{n}))$ when $\epsilon$ tends to 0.
\end{proof}

\begin{cor}\label{lem:gmea4}
Let there be given compact metrizable spaces $X$ and $Y$. Given any unital positive linear map $\phi$ from $C(X)$ to $C(Y)$ which sends $C(X)_{(p,q,\alpha)}$ to $C(Y)_{(z,w,\beta)}$, denote by $\mu_{z}$, $\mu_{w}$ the measures induced by evaluations of $\phi$ at $z$ and $w$. Let there be given a partition $\{X_{1},X_{2},...,X_{n}\}$ of $X$, where $X_{i}$ is a connected Borel set and $p\in X_{1},q\in X_{n}$.
Then we have the following distribution of $\mu_z$ and $\mu_w$ with respect to the partition:
$$\begin{aligned}
&\mu_{z}(X_{1})=\beta\mu_{w}(X_{1})+\frac{1-\beta}{1-\alpha},\\
&\mu_{z}(X_{n})=\beta\mu_{w}(X_{n})-\frac{1-\beta}{1-\alpha}\alpha.\\
\end{aligned}$$
In particularly, if $\alpha=\beta$, then $\phi(f)(z)=f(p), \phi(f)(w)=f(q)$ for all $f\in C(X)$.
\end{cor}

\textbf{Finally Theorem \ref{thm3} and Theorem \ref{thm4} are confirmed:}
\begin{proof}
Based on Lemma \ref{mea4} and Corollary \ref{lem:gmea4}, similar proofs as in Theorem \ref{thm1}, \ref{thm2} work.
\end{proof}

\proof[Acknowledgements] The research of the first author was supported by a
grant from the Natural Sciences and Engineering Research
Council of Canada. He is indebted to The Fields Institute
for their generous support of the field of operator algebras. The first author is indebted to Cristian Ivanescu
for substantial discussions. The second author is supported by the National Science Foundation of China with grant No. 11501060; he is also supported by the Fundamental Research Funds for the Central Universities (Project No. 2018CDXYST0024 in Chongqing University).


\begin{thebibliography}{}



\bibitem{G} G. Elliott, \emph{The classification problem for amenable C*-algebras,} Proc. Internat. Congr. Math. (Zurich, 1994), Birkhauser, Basel, 1995, 922-932.


\bibitem{JS1} X. Jiang and H. Su, \emph{A Classification of simple limits of splitting interval algebras,} J. Funct. Anal. 151 (1997), 50-76.

\bibitem{JS2} X. Jiang and H. Su, \emph{On a simple unital projectionless C*-algebra, } Amer. J. Math. 121 (1999), no. 2, 359-413.

\bibitem{Li} L. Li, \emph{Simple inductive limit C*-algebras: spectra and approximations by interval algebras,} J. Reine Angew. Math. 507 (1991), 57-79.

\bibitem{Rak} S. Razak, \emph{On the classification of simple stably projectionless C*-algebras,} Canad. J. Math. 54(1) (2002), 138-224.

\bibitem{Est} E. Stormer, \emph{Positive linear maps of operator algebras,} Acta Math. 110 (1963), 233-278.

\bibitem{Th} K. Thomsen, \emph{Inductive limits of interval algebras: the tracial state space,} Amer. J. Math. 116 (1994), no. 3, 605-620.

\end{thebibliography}
\end{document}